\documentclass[11pt,reqno]{amsart}

\usepackage{amsmath,amssymb,amsthm,mathtools}
\usepackage{enumitem}
\usepackage{microtype}
\usepackage[colorlinks=true,linkcolor=blue,citecolor=blue,urlcolor=blue]{hyperref}
\usepackage{cleveref}
\numberwithin{equation}{section}

\theoremstyle{plain}
\newtheorem{thm}{Theorem}[section]

\usepackage[numbers,sort&compress]{natbib}

\theoremstyle{definition}

\newtheorem{assump}[thm]{Assumption}

\theoremstyle{remark}

\newtheorem{theorem}[thm]{Theorem}
\newtheorem{lemma}[thm]{Lemma}
\newtheorem{proposition}[thm]{Proposition}

\theoremstyle{definition}
\newtheorem{definition}[thm]{Definition}
\theoremstyle{remark}
\newtheorem{remark}[thm]{Remark}

\newtheorem{thmA}{Theorem}

\newtheorem{thmB}{Theorem}

\newtheorem{thmC}{Theorem}

\newtheorem{thmD}{Theorem}

\newcommand{\R}{\mathbb{R}}
\newcommand{\II}{[0,1]}
\newcommand{\Wpm}{W^{1,2}_{x_\pm}}
\newcommand{\Wz}{W^{1,2}_0}
\newcommand{\Lc}{\mathcal{L}}
\newcommand{\Oc}{\mathcal{OM}}
\newcommand{\Wc}{\mathcal{W}}
\newcommand{\Ac}{\mathcal{A}}

\newcommand{\Pc}{\mathcal{P}}
\newcommand{\Kc}{\mathcal{K}}
\newcommand{\fa}{\mathfrak{a}}

\newcommand{\angles}[2]{\langle #1,#2\rangle}
\DeclareMathOperator{\sfl}{sf}
\DeclareMathOperator{\sign}{sign}
\DeclareMathOperator{\Spm}{Sp}

\DeclareMathOperator{\ran}{ran}

\newcommand{\mm}{m^-}
\newcommand{\mpplus}{m^+}

\newcommand{\indic}{\mathbf{1}}

\title[Morse index and spectral flow for the Onsager--Machlup action]{Noise-intensity bifurcations of transition paths by Morse index formula}

\author{Zhihao Zhao}
\address{School of Mathematics, China University of Mining and Technology,\\
Xuzhou 221000, China}
\email{zhaozh@cumt.edu.cn}
\author{Jinqiao Duan}
\address{School of Sciences, Great Bay University, Dongguan 523000, China}
\email{duan@gbu.edu.cn}



\begin{document}

\begin{abstract}
	We study most probable transition paths (MPTPs) of a gradient diffusion system in $\R^n$ with fixed endpoints, in the framework of Onsager--Machlup (OM) theory. The existence theory for such paths is developed within Lagrangian variational theory: above the Ma\~n\'e critical value $c_u(L)$ the energy-penalized action admits global minimizers, while below the endpoint-dependent critical value $k_0(L;x_\pm)$ only time-truncated minimizers exist; at regular levels between them, free-time extremals exist.  The main result is a $\{0,1\}$-index theorem: at a nondegenerate free-time extremal, the Morse index of the full Hessian equals the fixed-time Morse index plus a correction $\nu\in\{0,1\}$, equal to $1$ exactly when the reduced (minimal fixed-time) action satisfies $S''(T)<0$. The loss of local minimality along a branch is thereby split into two mechanisms: a conjugate-point mechanism and a duration mechanism.  Combined with a Hamiltonian spectral-flow formula, this yields a bifurcation criterion in the noise intensity $\sigma$ and a sharp local stability criterion expressed in terms of the crossing instants.  As an example, the one-dimensional quartic double well is analyzed in detail.
\end{abstract}

\maketitle

\section{Introduction}\label{sec:intro}
The Onsager--Machlup (OM) theory is a fundamental tool of non-equilibrium statistical physics and of the theory of stochastic processes, especially for characterizing the most probable
paths of noise-driven dynamical systems.  For the gradient diffusion
\begin{equation}\label{eq:sde}
dX_t=-\nabla V(X_t)\,dt+\sqrt{2\sigma}\,dW_t
\end{equation}
in $\R^n$ with noise intensity $\sigma>0$, the probability that a realisation of \eqref{eq:sde} remains in a narrow tube around a given smooth path $\gamma$ is, to leading exponential order in the tube radius, governed by the OM functional
\begin{equation}\label{eq:om-intro}
\Oc(\gamma)=\int_0^T\Big[\tfrac12\big|\dot\gamma(t)+\nabla V(\gamma(t))\big|^2-\sigma\Delta V(\gamma(t))\Big]dt;
\end{equation}
see the foundational works \cite{onsager_fluctuations_1953,durr_onsager-machlup_1978,takahashi_probability_1981} and, e.g., \cite{taniguchi_onsager-machlup_2007} for further developments.  Given two metastable states $x_-,x_+$, a minimizer of \eqref{eq:om-intro} over paths connecting $x_-$ to $x_+$ (and over the duration $T$) is called a \emph{most probable transition path} (MPTP).  MPTPs are used to quantify rare events and transition mechanisms in stochastic dynamics, complementing the minimum-action approaches of Freidlin--Wentzell theory \cite{freidlin_wentzell}.

\subsection{The variational problem}
Expanding the square in \eqref{eq:om-intro} produces the mechanical Lagrangian $L(x,v)=\tfrac12|v|^2-U(x)$ with the $\sigma$-dependent effective potential
\begin{equation}\label{eq:U-intro}
U(x)=\sigma\Delta V(x)-\tfrac12|\nabla V(x)|^2,
\end{equation}
and the MPTP problem becomes a free-time (free-duration) variational problem; see \eqref{eq:freeaction}.  

A first issue is that the double minimization $\inf_{T>0}\inf_\gamma\Oc(\gamma)$ is ill-posed in some situations of physical interest: if $V$ is bounded from below with two nondegenerate local minima (\ref{ass-h2}--\ref{ass-h3} below), then this free-time OM action is unbounded below, with infimum $-\infty$ (cf.\ \cite[Proposition~3]{du_graph_2021}).  Numerically, finite-time OM minimizers are used as approximations of the MPTP; there are also geometric minimizers associated with a geometric OM functional, which may be recovered by time-parametrized MPTPs; see \cite{heymann_vanden-eijnden_2008,SunZhou2018,grafke_vanden-eijnden_2019}.  Analytically, the existence of minimizers depends on the choice of the Hamiltonian energy level.  From the point of view of Lagrangian functional theory, the MPTP is characterized as a minimizer of a free-time Lagrangian action functional: as an extremal curve its Hamiltonian energy vanishes identically, while as a minimizer this energy level should exceed the Ma\~n\'e critical value $c_u(L)$; see, e.g., \cite{mane_lagrangian_1997-1,contreras_lagrangian_1997-1,asselle_existence_2016}.  A double-well system with $c_u(L)>0$ therefore admits no such minimizer---another way of seeing that the free-time action (the OM functional) must be unbounded below.
In this situation two regularizations are available: (i) the \emph{time-truncated} functional $\Lc^{\leq\tau}$, in which the duration is constrained to $(0,\tau]$ and the functional is bounded from below; and (ii) the \emph{energy-penalized} functional $\Lc_k=\Lc+kT$ with $k>c_u(L)$, which is coercive and whose global minimizers exist and can be regarded as graphs approximating geometric MPTPs.  This principle underlies several numerical approaches to the computation of MPTPs, such as \cite{du_graph_2021,li_gamma-limit_2021,e_vanden-eijnden_2010}.  For the Dirichlet boundary value problem considered here, the existence of such orbits is further constrained by the endpoints $x_\pm$: this introduces a second critical value $k_0(L;x_\pm)$, the minimal energy level at which $x_-$ and $x_+$ lie in the same sublevel component of the potential, and hence a lower bound on the energy of transition paths; see \eqref{eq:k0}.  Under the assumption $k_0(L;x_\pm)<k$, we study \emph{critical points} of these functionals $\mathcal L_k$ and their \emph{local} minimality.

Bifurcations of MPTPs have been studied in the weak-noise regime described by Freidlin--Wentzell theory; see, e.g., \cite{maier_scaling_1996,borner_saddle_2024}.  In that setting the most probable transition typically follows a heteroclinic orbit of the deterministic skeleton; see \cite{dykman_observable_1994,fleurantin_dynamical_2023}.  Bifurcations of the MPTP are then directly associated with phase transitions of the stochastic system, so that a bifurcation point of a model parameter indicates a critical transition.

In contrast to the weak-noise regime, this paper investigates bifurcation phenomena within OM theory, where the noise intensity $\sigma$ is a central parameter: the OM functional, and hence the MPTP, depends explicitly on $\sigma$, and it is natural to ask how the transition path bifurcates as $\sigma$ varies.  Related works have approached this question from different angles, but with methods distinct from ours.  Numerically, changes of the most probable path with the noise intensity were studied through bifurcation diagrams and the OM-potential landscape in \cite{hiruta_most_2026}, and the effect of the noise intensity on OM transition pathways was computed for a genetic regulatory network in \cite{hu_onsagermachlup_2022}.  In \cite{fleurantin_dynamical_2023}, the Maslov index is used to identify local minimizers of the action, and the OM functional then selects the most probable escape path among them at finite noise.  Analytically, the tools closest to ours are the spectral-flow and crossing-form bifurcation criteria of \cite{waterstraat_spectral_2015} (for homoclinics of Hamiltonian systems) and of \cite{piccione_spectral_2004} (bifurcation of semi-Riemannian geodesics at conjugate points).  The bifurcations studied in this paper are variational: a change of $\sigma$ causes the MPTP, as a critical point of the OM functional, to lose local minimality---and, in the double-well case study of Section~\ref{sec:examples}, its status as the global minimizer.

The central object of this paper is the \emph{Morse index} of the second variation of the OM action, which is the path-space analogue of the unstable-mode count in classical rate theory.  In the Eyring--Kramers theory \cite{eyring_1935,kramers_1940,bovier_etal_2004,berglund_2013} and in Langer's extension \cite{langer_1969}, the transition rate is governed by the Hessian of the potential at the relevant saddle, and the number of unstable directions is exactly one. On path space, the MPTP plays a role analogous to that of the saddle in classical rate theory, and the role of the unstable-mode count is played by the Morse index of the OM action at the MPTP: changes of the index mark qualitative changes of the tube-fluctuation structure around the most probable path. 

\subsection{Main results}
We now describe the main results informally. Precise statements are given in the sections indicated.

\begin{thmA}[Energy stratification and existence; Theorem~\ref{thm:A}]\label{thm:intro-A}
Let $x_-\neq x_+$ and assume conditions \ref{ass-h1},\ref{ass-h2}.  \emph{(i)} If $k>c_u(L)$, the penalized functional $\Lc_k$ admits a global minimizer $(x_k,T_k)$ with $T_k>0$, which is a free-time critical point of $\Lc_k$ lying on the energy level $E\equiv k$.  \emph{(ii)} If $k<k_0(L;x_\pm)$, then for every $\tau>0$ the truncated functional $\Lc_k^{\leq\tau}$ admits a global minimizer, necessarily attained at $T=\tau$.  \emph{(iii)} If $k>k_0(L;x_\pm)$ and the level $k$ is regular in the sense of Theorem~\ref{thm:A}(iii), then there exists a free-time extremal with $E\equiv k$.
\end{thmA}
Let $(x,T)$ be a free-time critical point (extremal) of $\Lc$ with energy $E\equiv0$, as provided by Theorem~A(iii) at the level $k=0$, nondegenerate in the sense that the associated second-variation operator $A$ (see \eqref{eq:sturm}) is invertible.  Let $S$ denote the reduced action---the minimal fixed-time action along the local branch of fixed-time critical points through $x$---and suppose $S''(T)\neq0$.  Our next result relates the free-time Morse index $\mm(x,T)$ of \eqref{eq:morse} to the fixed-time Morse index $\mm_T(x)$ of \eqref{eq:morsefixed}.

\begin{thmB}[The $\{0,1\}$-index; Theorem~\ref{thm:B}]\label{thm:intro-B}
	There holds
	\begin{equation*}
		\mm(x,T)=\mm_T(x)+\nu,\qquad \nu=\indic_{\{S''(T)<0\}}\in\{0,1\},
	\end{equation*}
and equivalently $\nu=\indic_{\{C_0<\angles{\beta}{A^{-1}\beta}\}}$ with $\beta=-2\nabla U(x(\cdot))$ and $C_0=T^{-3}\int_0^1|\dot x|^2dt$. 
\end{thmB}
 
 The correction $\nu$ measures the nonconvexity of the reduced action $S$ in the duration direction; note that it is generated by the $\sigma$-dependent landscape $U$, i.e.\ by the OM-specific term $-\sigma\int\Delta V\,dt$ in \eqref{eq:om-intro}.  To turn changes of the index into bifurcation statements as the noise intensity $\sigma$ varies, we split the spectral flow of the full Hessian according to the two terms in Theorem~B.
 
\begin{thmC}[Bifurcation formula; Theorem~\ref{thm:C}]\label{thm:intro-C}
Along a $C^1$ branch $(x_\sigma,T_\sigma)$ of free-time critical points with nondegenerate endpoints,
\[
\sfl(I_\sigma(0),\sigma\in[\sigma_0,\sigma_1])
=\sfl(\fa(\sigma),\sigma\in[\sigma_0,\sigma_1])+\sfl(\Ac(\sigma,1),\sigma\in[\sigma_0,\sigma_1]),
\]
where $\fa(\sigma)=S_\sigma''(T_\sigma)$ is the scalar curvature of the reduced action and $\Ac(\sigma,s)=-J\frac{d}{dt}-B_{\sigma,s}$ is the Hamiltonian operator of the linearized Euler--Lagrange system. 
\end{thmC}

Under the Fitzpatrick--Pejsachowicz--Recht hypotheses (Theorem~\ref{thm:fpr}), nonzero spectral flow implies the \emph{existence of a variational bifurcation point} in $[\sigma_0,\sigma_1]$.  The bifurcation criterion is thus reduced to detecting nonzero spectral flow of $\fa$ and of $\Ac$: the scalar path $\fa(\sigma)=S_\sigma''(T_\sigma)$ is read off directly from the reduced action, so it remains to analyze the family $\{\Ac(\sigma,1)\}$.  A necessary condition for this spectral flow to be nonzero is that $\ker\Ac(\sigma_*,1)\neq\{0\}$ for some $\sigma_*\in[\sigma_0,\sigma_1]$; we therefore study the crossing set $\{(\sigma,s):\ker\Ac(\sigma,s)\neq\{0\}\}$ and its local structure near a crossing at $s=1$.

\begin{thmD}[Local stability criterion; Theorem~\ref{thm:D}]\label{thm:intro-D}
Let $\{x_\sigma\}$ be a $C^1$ branch of fixed-time critical points such that $x_{\sigma_0}$ is a local minimizer and $\dim\ker\Ac(\sigma_0,1)=1$.  The crossing set is locally a $C^1$ graph $s(\sigma)$ with
\[
s'(\sigma_0)=-\frac{\Gamma(\Ac,\sigma)[u]}{\Gamma(\Ac,s)[u]},\qquad u\in\ker\Ac(\sigma_0,1),
\]
and for small $\delta>0$: if $s'(\sigma_0)>0$, then $x_\sigma$ remains a local minimizer for $\sigma\in(\sigma_0,\sigma_0+\delta)$; if $s'(\sigma_0)<0$, then $x_\sigma$ acquires exactly one negative direction (loss of local minimality).
\end{thmD}

The results above are illustrated in Section~\ref{sec:examples} on the one-dimensional quartic double well $V(x)=(x^2-1)^2/4$: \emph{(i)} increasing $\sigma$ lowers the conjugate-point threshold of the direct fixed-time MPTP---a noise-induced loss of local minimality, proved by a crossing-form computation along the straight-path approximation and confirmed numerically for the true extremal branch (Section~\ref{subsec:dw}); \emph{(ii)} at horizon $\tau=3$ the complete fixed-time bifurcation diagram exhibits a supercritical pitchfork at $\sigma_c=0.1628$, at which the loss of local minimality of the symmetric MPTP coincides with the switching of the \emph{global} minimizer to a mirror pair of asymmetric detour paths (Proposition~\ref{prop:pitchfork}); \emph{(iii)} for endpoints inside $\{U_\sigma<0\}$ the $\{0,1\}$-alternative occurs with $\nu=1$: the long zero-energy arc is a fixed-time local minimizer but a free-time saddle of index one (Proposition~\ref{prop:dw01}), and the two terms $(\mm_T,\nu)$ are observed to exchange with the sum conserved (Proposition~\ref{prop:dwconv}).

\subsection{Organization}
Section~\ref{sec:setup} introduces the Lagrangian formulation, the two critical values $c_u(L)$ and $k_0(L;x_\pm)$, and Theorem~A.  Section~\ref{sec:second} computes the fixed-time and free-time second variations and sets up the regularized Fredholm framework.  Section~\ref{sec:01} proves Theorem~B.  Section~\ref{sec:sf} develops the Hamiltonian spectral-flow formalism and proves Theorem~C.  Section~\ref{sec:stab} proves Theorem~D and its free-time counterpart.  Section~\ref{sec:examples} analyzes the quartic double well in detail: the crossing-curve picture for the conjugate threshold, the complete fixed-time bifurcation diagram at $\tau=3$, and a numerical realization of the $\{0,1\}$-alternative.   Appendix~\ref{app:ps} collects the Palais--Smale details for Theorem~A; Appendix~\ref{app:crossing} gives the boundary representation of crossing forms.

\section{Lagrangian formulation and energy stratification}\label{sec:setup}

In this section we set up the variational framework, define the two critical values that organize the theory, and prove the existence theorem: Theorem~A.

\subsection{The Onsager--Machlup functional as a Lagrangian action}\label{subsec:om}

Let $V:\R^n\to\R$ be a potential.  Throughout the paper we assume the following property.

\begin{assump}[Property (H)]\label{assum:H}
\leavevmode
\begin{enumerate}[label=\textup{(H\arabic*)},series=H]
\item \label{ass-h1} $V\in C^4(\R^n,\R)$;
\item\label{ass-h2} the function
\begin{equation}\label{eq:U}
U(x):=\sigma\Delta V(x)-\tfrac12|\nabla V(x)|^2
\end{equation}
is bounded from above, and $U(x)\to-\infty$ as $|x|\to+\infty$. Then maximum points are contained in a bounded domain;
\item \label{ass-h3} $V$ has two nondegenerate local minima $x_0,x_1$: $\nabla V(x_i)=0$ and $\nabla^2V(x_i)>0$, $i=0,1$.
\end{enumerate}
\end{assump}

Parts \ref{ass-h1},\ref{ass-h2} are the standing regularity/coercivity assumptions; part \ref{ass-h3} is imposed only when the bistable situation is discussed.  Note that \ref{ass-h3} implies $U(x_i)=\sigma\Delta V(x_i)>0$.

For a time horizon $T>0$ and an absolutely continuous path $\gamma:[0,T]\to\R^n$, the Onsager--Machlup functional of the diffusion \eqref{eq:sde} is
\begin{equation}\label{eq:om}
\Oc(\gamma)=\int_0^T\Big[\tfrac12\big|\dot\gamma(t)+\nabla V(\gamma(t))\big|^2-\sigma\Delta V(\gamma(t))\Big]dt.
\end{equation}
Given two endpoints $x_-,x_+\in\R^n$, write
\[
C_\pm[0,T]:=\{\gamma\in AC([0,T],\R^n):\gamma(0)=x_-,\ \gamma(T)=x_+\}.
\]
Expanding the square in \eqref{eq:om} and integrating the cross term gives
\begin{equation}\label{eq:expansion}
\begin{aligned}
\Oc(\gamma)
&=\int_0^T\Big[\tfrac12|\dot\gamma(t)|^2+\angles{\dot\gamma(t)}{\nabla V(\gamma(t))}-U(\gamma(t))\Big]dt\\
&=\int_0^T\Big[\tfrac12|\dot\gamma(t)|^2-U(\gamma(t))\Big]dt+V(x_+)-V(x_-).
\end{aligned}
\end{equation}
The boundary term $V(x_+)-V(x_-)$ is constant on $C_\pm[0,T]$ and is therefore dropped without loss of generality.  We are thus led to the mechanical Lagrangian
\begin{equation}\label{eq:L}
L(x,v)=\tfrac12|v|^2-U(x),
\end{equation}
and to the action $\int_0^T L(\gamma,\dot\gamma)\,dt$.  The associated energy function and (Legendre-dual) Hamiltonian are
\begin{equation}\label{eq:energy}
E(x,v):=\angles{\partial_vL(x,v)}{v}-L(x,v)=\tfrac12|v|^2+U(x),\qquad
H(q,p)=\tfrac12|p|^2+U(q).
\end{equation}

We parametrize all paths on the unit interval $[0,1]$.  Let
\[
\Wpm(\II,\R^n):=\{x\in W^{1,2}(\II,\R^n)\ |\ x(0)=x_-,\ x(1)=x_+\},
\]
a smooth Hilbert manifold (the preimage of $(x_-,x_+)$ under the submersive evaluation map $x\mapsto(x(0),x(1))$), whose tangent space at $x$ is
\[
\Wz(\II,\R^n)=\{\xi\in W^{1,2}(\II,\R^n)\ |\ \xi(0)=0=\xi(1)\}.
\]
There is a one-to-one correspondence between pairs
\begin{equation}\label{eq:corr}
(x,T)\in\Wc:=\Wpm(\II,\R^n)\times\R^+
\qquad\text{and}\qquad
\gamma\in\bigcup_{T>0}C_\pm[0,T]\cap W^{1,2},
\end{equation}
given by $\gamma(t)=x(t/T)$, where $t\in[0,T]$ is the physical time.  Throughout, a dot denotes differentiation with respect to the \emph{normalized} time $t\in\II$.  We endow $\Wc$ with the product metric.  Define the free-time action functional
\begin{equation}\label{eq:freeaction}
\Lc(x,T):=T\int_0^1 L\Big(x(t),\frac{\dot x(t)}{T}\Big)dt=\Oc(\gamma),
\end{equation}
and, for $k\in\R$, the energy-penalized functional
\begin{equation}\label{eq:Lk}
\Lc_k(x,T):=T\int_0^1\Big[L\Big(x(t),\frac{\dot x(t)}{T}\Big)+k\Big]dt=\Lc(x,T)+kT.
\end{equation}
It is standard that $\Lc_k$ is smooth on $\Wc$ (see, e.g., \cite[Section~3]{abbondandolo_smooth_2009-3}).  For $\tau>0$ we also consider the \emph{time-truncated} functional
\begin{equation}\label{eq:trunc}
\Lc_k^{\leq\tau}:=\Lc_k\big|_{\Wpm(\II,\R^n)\times(0,\tau]},
\end{equation}
and the \emph{fixed-time} functional $\Lc^T:=\Lc(\cdot,T)$ on $\Wpm(\II,\R^n)$.  We omit the subscript $0$ in $\Lc_0$, $\Lc_0^T$ when no confusion arises.

\subsection{First variation and the energy identity}\label{subsec:firstvar}

A pair $(x,T)\in\Wc$ is a critical point of $\Lc$ if $d\Lc(x,T)[(\xi,b)]=0$ for every variation $(\xi,b)\in\Wz(\II,\R^n)\times\R$.

\begin{proposition}[First variation]\label{prop:firstvar}
$(x,T)$ is a critical point of $\Lc$ if and only if
\begin{equation}\label{eq:firstvar}
\begin{cases}
\dfrac{d}{dt}\,\partial_vL\Big(x(t),\dfrac{\dot x(t)}{T}\Big)=T\,\partial_xL\Big(x(t),\dfrac{\dot x(t)}{T}\Big),\\[2.2ex]
\angles{\partial_vL\Big(x(t),\dfrac{\dot x(t)}{T}\Big)}{\dfrac{\dot x(t)}{T}}-L\Big(x(t),\dfrac{\dot x(t)}{T}\Big)=0.
\end{cases}
\end{equation}
The first equation is the Euler--Lagrange equation in normalized time; the second is the energy condition $E(x,\dot x/T)\equiv0$.  More generally,
\begin{equation}\label{eq:dT}
\frac{\partial\Lc_k(x,T)}{\partial T}=\int_0^1\Big[k-E\Big(x(t),\frac{\dot x(t)}{T}\Big)\Big]dt,
\end{equation}
and the energy $E(x(t),\dot x(t)/T)$ is constant in $t$ along every solution of the Euler--Lagrange equation; hence a critical point of $\Lc_k$ satisfies $E\equiv k$.
\end{proposition}

\begin{proof}
Differentiating \eqref{eq:freeaction} in a direction $(\xi,b)$ gives
\[
\begin{aligned}
d\Lc(x,T)[(\xi,b)]=\;&T\int_0^1\Big[\angles{\partial_xL}{\xi}+T^{-1}\angles{\partial_vL}{\dot\xi}\Big]dt\\
&+b\int_0^1\Big[L-T^{-1}\angles{\partial_vL}{\dot x}\Big]dt,
\end{aligned}
\]
with $L$ and its derivatives evaluated at $(x,\dot x/T)$.  Vanishing of the $\xi$-part for all $\xi\in\Wz$ is equivalent to the first line of \eqref{eq:firstvar} (plus regularity of $x$, which follows from the strict convexity of $L$ in $v$; cf.\ \cite[Corollary~2.2.11]{fathi_weak_2008}).  Vanishing of the $b$-part gives the second line. This also proves \eqref{eq:dT}, since $L-\angles{\partial_vL}{\dot x/T}=-E(x,\dot x/T)$.  Finally, along an Euler--Lagrange orbit,
\[
\frac{d}{dt}E\Big(x,\frac{\dot x}{T}\Big)
=\Big\langle \dot x,\ \frac{1}{T^2}P(t)\ddot x+\frac1T Q^\top(t)\dot x-\bar L(t)\Big\rangle=0,
\]
where the second factor vanishes by \eqref{eq:eleexpanded} below; here $P,Q,\bar L$ are as in \eqref{eq:pqr}.  Hence $E$ is constant, and \eqref{eq:dT} forces the constant to be $k$ at a critical point of $\Lc_k$.
\end{proof}

For the OM Lagrangian \eqref{eq:L} the system \eqref{eq:firstvar} reads, in physical time $\gamma(t)=x(t/T)$,
\begin{equation}\label{eq:elphysical}
\ddot\gamma(t)=-\nabla U(\gamma(t)),\qquad \tfrac12|\dot\gamma(t)|^2+U(\gamma(t))\equiv0,
\end{equation}
or equivalently, in normalized time,
\begin{equation}\label{eq:elscaled}
\ddot x(t)=-T^2\nabla U(x(t)),\qquad \tfrac12\Big|\frac{\dot x(t)}{T}\Big|^2+U(x(t))\equiv0.
\end{equation}
With the notation
\begin{equation}\label{eq:pqr}
P(t)=\partial_{vv}L\Big(x,\frac{\dot x}{T}\Big),\qquad
Q(t)=\partial_{xv}L\Big(x,\frac{\dot x}{T}\Big),\qquad
R(t)=\partial_{xx}L\Big(x,\frac{\dot x}{T}\Big),
\end{equation}
and $Q^\top(t)=\partial_{vx}L(x,\dot x/T)$, $\bar L(t)=\partial_xL\Big(x,\frac{\dot x}{T}\Big)$, the Euler--Lagrange equation expands as
\begin{equation}\label{eq:eleexpanded}
\frac{1}{T^2}P(t)\ddot x(t)+\frac1T Q^\top(t)\dot x(t)=\bar L(t).
\end{equation}

\subsection{Two critical values}\label{subsec:critvals}

The existence of connecting orbits is organized by two critical values.

\begin{definition}[Ma\~n\'e critical value]\label{defi:mane}
The Ma\~n\'e critical value of $L$ is
\begin{equation}\label{eq:mane}
c_u(L):=\inf\Big\{k\in\R:\int_0^T\big[L(\gamma,\dot\gamma)+k\big]dt\geq0\ \ \forall\,\gamma\in\bigcup_{T>0}C([0,T],\R^n)\Big\}.
\end{equation}
\end{definition}

\begin{proposition}\label{prop:cu}
If $L$ is given by \eqref{eq:L} and $U$ satisfies \ref{ass-h2}, then
\begin{equation}\label{eq:cu=supU}
c_u(L)=\sup_{x\in\R^n}U(x).
\end{equation}
\end{proposition}

\begin{proof}
Since $L(x,v)=\tfrac12|v|^2-U(x)\geq-\sup U$, for every $k\geq\sup U$ and every path $\gamma$ one has $\int_0^T[L+k]dt\geq T(k-\sup U)\geq0$, so $c_u(L)\leq\sup U$.  Conversely, let $x^*$ be a maximum point of $U$ (which exists by \ref{ass-h2}) and let $\gamma(t)\equiv x^*$ be the constant path on $[0,1]$: then $\int_0^1[L+k]dt=-U(x^*)+k$, which is negative for every $k<\sup U$.  Hence $c_u(L)\geq\sup U$.
\end{proof}

\begin{remark}\label{rmk:maneclosed}
For mechanical Lagrangians the value \eqref{eq:cu=supU} coincides with the usual Ma\~n\'e critical value defined through closed loops (or through the critical value of the universal cover): the constant loops at maximum points of $U$ already takes the infimum.  Ma\~n\'e critical values also admit an equivalent Hamiltonian definition; see \cite{mane_lagrangian_1997-1,contreras_lagrangian_1997-1} and see \cite[Remark~1.4]{asselle_existence_2016} for critical values of boundary value problems.
\end{remark}

\begin{definition}[Endpoint critical value]\label{defi:k0}
The endpoint critical value of the pair $(x_-,x_+)$ is
\begin{equation}\label{eq:k0}
k_0(L;x_\pm):=\inf\big\{k\in\R:\ x_-,x_+\text{ lie in the same component of }\{U\leq k\}\big\}.
\end{equation}
\end{definition}

For the Hamiltonian $H(q,p)=\tfrac12|p|^2+U(q)$, the level set $H^{-1}(k)$ projects onto $\{U\leq k\}$, and the fibre over $x$ is nonempty exactly when $U(x)\leq k$; hence \eqref{eq:k0} is precisely the \emph{mountain-pass level} of $U$ between $x_-$ and $x_+$: the infimum of $k$ such that $x_\pm$ can be joined inside $\{U\leq k\}$.  In particular
\begin{equation}\label{eq:k0geq}
k_0(L;x_\pm)\geq\max\{U(x_-),U(x_+)\},
\end{equation}
and the inequality may be strict.

\begin{lemma}[Necessary condition]\label{lem:necessary}
If there exists an Euler--Lagrange orbit $\gamma$ connecting $x_-$ to $x_+$ with energy $E(\gamma,\dot\gamma)\equiv k$, then $k\geq k_0(L;x_\pm)$.  Consequently a free-time critical point of $\Lc_k$ connecting $x_\pm$ can exist only if $k\geq k_0(L;x_\pm)$; in particular the free-time problem for $\Lc=\Lc_0$ requires $k_0(L;x_\pm)\leq0$.
\end{lemma}

\begin{proof}
Along the orbit, $\tfrac12|\dot\gamma|^2+U(\gamma)=k$, hence $U(\gamma(t))\leq k$ for all $t$; the image of $\gamma$ is then a connected subset of $\{U\leq k\}$ containing both endpoints, so $x_\pm$ lie in the same component of $\{U\leq k\}$ and $k\geq k_0(L;x_\pm)$ by definition.  The remaining statements follow from Proposition~\ref{prop:firstvar}.
\end{proof}

\subsection{Proof of Theorem~A}\label{subsec:thmA}

We recall the Palais--Smale (PS) condition: a sequence $\{(x_n,T_n)\}\subset\Wc$ is a PS sequence for $\Lc_k$ at level $c$ if $\Lc_k(x_n,T_n)\to c$ and $\|d\Lc_k(x_n,T_n)\|\to0$; $\Lc_k$ satisfies (PS) if every PS sequence has a convergent subsequence.  The following abstract fact is standard (see, e.g., \cite{abbondandolo_lectures_2013-1}): \emph{if $f$ is a $C^1$ functional on a Hilbert manifold which is bounded from below, has complete sublevels, and satisfies \emph{(PS)}, then $f$ admits a global minimizer, which is a critical point.}

The key estimate is the following two-sided bound on the duration.  Let
\[
l(x_\pm):=\inf\Big\{\Big(\int_0^1|\dot x|^2dt\Big)^{1/2}\ |\ \ x\in\Wpm(\II,\R^n)\Big\}=|x_+-x_-|.
\]

\begin{lemma}[Coercivity estimate and completeness of sublevels]\label{lem:complete}
Assume $x_-\neq x_+$ and \emph{(H1)--(H2)}.  For every $(x,T)\in\Wc$,
\begin{equation}\label{eq:twosided}
\Lc_k(x,T)\geq T\int_0^1\Big[\frac{|\dot x(t)|^2}{2T^2}-c_u(L)+k\Big]dt
\geq\frac{l(x_\pm)^2}{2T}+T\big(k-c_u(L)\big).
\end{equation}
Consequently:
\begin{enumerate}
\item[(i)] if $k>c_u(L)$, then $\Lc_k$ is bounded from below on $\Wc$, and on every sublevel set $[\Lc_k\leq c]$ the durations are bounded away from $0$ and from $+\infty$; in particular each sublevel set is complete in $\Wc$;
\item[(ii)] for every $\tau>0$ and $k\in\R$, $\Lc_k^{\leq\tau}$ is bounded from below by $\tfrac{l(x_\pm)^2}{2\tau}-\tau\,|k-c_u(L)|$, and its sublevel sets are complete in $\Wpm(\II,\R^n)\times(0,\tau]$.
\end{enumerate}
\end{lemma}

\begin{proof}
The first inequality in \eqref{eq:twosided} uses $-U\geq-c_u(L)$ (Proposition~\ref{prop:cu}); the second uses $\int_0^1|\dot x|^2dt\geq l(x_\pm)^2$.  (i) If $k>c_u(L)$, the function $T\mapsto l^2/(2T)+T(k-c_u)$ tends to $+\infty$ as $T\to0^+$ and as $T\to+\infty$ and has a positive minimum, so $[\Lc_k\leq c]$ is contained in a slab $\Wpm\times[T_{\min},T_{\max}]$ with $T_{\min}>0$. Since $l(x_\pm)>0$ when $x_-\neq x_+$, the limit $T=0$ is excluded and the sublevel is complete.  (ii) On $(0,\tau]$ the same estimate gives the lower bound and keeps $T$ away from $0$ on sublevels. The truncation $(0,\tau]$ is closed in $\R^+$, so the sublevels are complete.
\end{proof}

\begin{lemma}[Palais--Smale condition]\label{lem:ps}
Assume $x_-\neq x_+$ and \emph{(H1)--(H2)}.
\begin{enumerate}
\item[(i)] If $k>c_u(L)$, every \emph{(PS)} sequence $\{(x_n,T_n)\}$ of $\Lc_k$ has a subsequence converging to some $(x_*,T_*)\in\Wc$ with $T_*>0$.
\item[(ii)] For every $\tau>0$ and $k\in\R$, every \emph{(PS)} sequence of $\Lc_k^{\leq\tau}$ has a subsequence converging to some $(x_*,T_*)$ with $0<T_*\leq\tau$.
\end{enumerate}
\end{lemma}

The proof follows the classical argument of \cite[Lemma~5.3]{abbondandolo_lectures_2013-1} (see also \cite[Lemma~2.5]{asselle_existence_2016}); the only new point is that the configuration space $\R^n$ is not compact.  This is patched by the two-sided estimate \eqref{eq:twosided}, which confines every (PS) sequence to a bounded interval $0<T_{\min}\leq T_n\leq T_{\max}<\infty$ and yields $\sup_n\|\dot x_n\|_{L^2}<\infty$. The path family is then $\tfrac12$-equi-H\"older, hence uniformly convergent up to a subsequence, so the paths eventually lie in a bounded ball $B_r\subset\R^n$ and the compact-manifold argument applies to $\tilde L:=L|_{TB_r}$.  The details are carried out in Appendix~\ref{app:ps}.

\bigskip
Note that $\{U\leq\sup U\}=\R^n$, we have $k_0(L;x_\pm)\leq c_u(L)$, and the equality holds if $\sup U=\max\{U(x_-),U(x_+)\}$.
\begin{theorem}[Energy stratification and existence]\label{thm:A}
Let $x_-\neq x_+$ and assume \ref{ass-h1},\ref{ass-h2}.
\begin{enumerate}
\item[(i)] \emph{(Supercritical energy.)} If $k>c_u(L)$, then $\Lc_k$ admits a global minimizer $(x_k,T_k)\in\Wc$ with $T_k>0$; it is a free-time critical point of $\Lc_k$ and the corresponding orbit $\gamma_k(t)=x_k(t/T_k)$ satisfies the Euler--Lagrange equation with energy $E\equiv k$.
\item[(ii)] \emph{(Subcritical endpoint level.)} If $k<k_0(L;x_\pm)$, then for every $\tau>0$ the truncated functional $\Lc_k^{\leq\tau}$ admits a global minimizer $(x^\tau,T^\tau)$, and necessarily $T^\tau=\tau$; the path $x^\tau$ is a critical point of the fixed-time functional $\Lc(\cdot,\tau)$, but the pair is in general \emph{not} a free-time critical point: $\partial_T\Lc_k=\int_0^1(k-E)\,dt\neq0$.
\item[(iii)] \emph{(Free-time extremals.) } Denote by $\Omega$ the connected component of $U(x)\leq k$ containing $x_\pm$. Suppose $x\in \{U(x)=k\}\cap \bar\Omega$ is regular in the sense $\nabla U(x)\neq 0$. Then $k>k_0(L;x_\pm)$ if and only if there exists a free-time extremal with $E\equiv k$.
\end{enumerate}
\end{theorem}

\begin{proof}
(i) By Lemma~\ref{lem:complete}(i), $\Lc_k$ is bounded below with complete sublevels. By Lemma~\ref{lem:ps}(i) it satisfies (PS). The minimization principle yields a global minimizer $(x_k,T_k)$ with $T_k>0$.  Being a free-time critical point of $\Lc_k$, Proposition~\ref{prop:firstvar} gives the Euler--Lagrange equation and $E\equiv k$.

(ii) By Lemma~\ref{lem:complete}(ii) and Lemma~\ref{lem:ps}(ii), the same principle applied on the closed set $\Wpm\times(0,\tau]$ yields a global minimizer $(x^\tau,T^\tau)$.  If $T^\tau<\tau$, then $(x^\tau,T^\tau)$ is an interior point of the time interval and hence a free-time critical point of $\Lc_k$, so $E\equiv k$ along the orbit; Lemma~\ref{lem:necessary} then forces $k\geq k_0(L;x_\pm)$, contradicting the hypothesis.  Therefore $T^\tau=\tau$.  Minimality in the $\Wpm$-factor shows that $x^\tau$ is a critical point of $\Lc(\cdot,\tau)$.  Since $k<k_0(L;x_\pm)$, the orbit has energy $E\equiv E_0\geq k_0(L;x_\pm)>k$, whence $\partial_T\Lc_k(x^\tau,\tau)=k-E_0\neq0$ and the pair is not a free-time critical point.

(iii) The implication $``\Leftarrow"$ has been proved in Lemma~\ref{lem:necessary}.  
Along a curve of energy $k$ one has $\tfrac12|\dot x|^2=k-U(x)$, hence $\Lc_k=\int_0^T[\tfrac12|\dot x|^2+k-U(x)]dt=\int_0^T\sqrt{2(k-U(x))}\,|\dot x|\,dt$. The action coincides with the length of the Jacobi metric $g_k=2(k-U)|dx|^2$, whose $g_k$-geodesics are exactly the Euler--Lagrange orbits of energy $k$. Now suppose $k>k_0(L;x_\pm)$. Since $U(x)\to-\infty$ as $|x|\to+\infty$ (see \ref{ass-h2}), $\{U<s\}^c$ is compact for sufficiently small $s<k_0(L;x_\pm)$, and so is $\bar\Omega\subset \{U<s\}^c$. So the infimum of the Jacobi length $\Lc_k$ is finite, and yields a length minimizer $\gamma$ in $\bar\Omega$. If $\gamma$ contact the boundary $\partial \bar\Omega$, the velocity $|\dot{x}|=\sqrt{2(k-U)}=0$, which is exactly the brake orbit; see an example in Section \ref{subsec:dwfree}. Thus $\gamma$ is a finite-duration extremal of energy $k$ from $x_-$ to $x_+$, identified with a free-time critical point of $\Lc_k$. The regularity condition $\nabla U(x)\neq 0$ for all $x\in\partial\bar\Omega$ rules out the possibility that the minimal geodesic degenerates into an infinite-duration heteroclinic chain.
\end{proof}

\begin{remark}\label{rmk:unbounded}
Under (H3) one has $c_u(L)\geq\sigma\Delta V(x_\pm)>0$, so $\Lc=\Lc_0$ is unbounded below on $\Wc$ (cf.\ \cite[Proposition~3]{du_graph_2021}): the double minimization $\inf_T\inf_\gamma\Oc(\gamma)$ has value $-\infty$ and no global minimizer exists; see also Theorem~\ref{thm:A}(i). In this case, the free-time theory below studies \emph{critical points} of $\Lc$ (energy $E\equiv0$), which by Lemma~\ref{lem:necessary} or Theorem \ref{thm:A}(iii) requires $k_0(L;x_\pm)\leq0$.
\end{remark}

We collect the standing assumptions used in the index theory for Sections \ref{sec:second}--\ref{sec:stab}:
\begin{enumerate}[label=\textup{(S\arabic*)},series=Sseries]
\item\label{ass-s1} \emph{Legendre convexity:} $\partial_{vv}L(x,v)\geq c\,I_n$ for some $c>0$.  (For the OM Lagrangian, $\partial_{vv}L=I_n$, so this is automatic.)
\item\label{ass-s2} For the \emph{free-time} theory, $k_0(L;x_\pm)\leq0$, so that zero-energy connections may exist (Lemma~\ref{lem:necessary}).
\item\label{ass-s3} Wherever a family $(x_\sigma,T_\sigma)$ of critical points is considered, it is assumed $C^1$ in $\sigma\in[\sigma_0,\sigma_1]$.
\end{enumerate}
All Morse indices below are finite under \ref{ass-s1}, since the second variation is a compact perturbation of a positive definite quadratic form; see Section~\ref{sec:second}.

\section{Second variation and Morse index}\label{sec:second}

Throughout this section, $(x,T)\in\Wc$ is a free-time critical point of $\Lc=\Lc_0$ (hence an Euler--Lagrange orbit of energy $E\equiv0$), and $P,Q,R,\bar L$ are defined in \eqref{eq:pqr}.  Under the hypothesis \ref{ass-s1}, $P(t)\geq cI_n$; for the OM Lagrangian, $P=I_n$, $Q=0$ and $R(t)=-\nabla^2U(x(t))$.

\subsection{The second variational analysis}\label{subsec:fixedsecond}

For $\xi,\eta\in\Wz(\II,\R^n)$, the second variation of the fixed-time functional $\Lc^T=\Lc(\cdot,T)$ at $x$ is
\begin{equation}\label{eq:secondfixed}
d^2\Lc^T(x)[\xi,\eta]
=\int_0^1\Big\langle-\frac{d}{dt}\Big[\frac1T P(t)\dot\xi+Q(t)\xi\Big]+Q^\top(t)\dot\xi+TR(t)\xi,\ \eta\Big\rangle dt,
\end{equation}
or, after integration by parts, the symmetric bilinear form
\begin{equation}\label{eq:secondfixedsym}
d^2\Lc^T(x)[\xi,\eta]
=\int_0^1\Big[\frac1T\angles{P\dot\xi}{\dot\eta}+\angles{Q\xi}{\dot\eta}+\angles{Q^\top\dot\xi}{\eta}+T\angles{R\xi}{\eta}\Big]dt.
\end{equation}
The associated \emph{Sturm--Jacobi operator} is the linearization of the Euler--Lagrange equation:
\begin{equation}\label{eq:sturm}
A\xi:=-\frac{d}{dt}\Big[\frac1T P\dot\xi+Q\xi\Big]+Q^\top\dot\xi+TR\xi,
\end{equation}
viewed either as a bounded self-adjoint operator on $\Wz(\II,\R^n)$ (via the form \eqref{eq:secondfixedsym}) or as an unbounded self-adjoint operator on $L^2(\II,\R^n)$ with domain $W^{2,2}\cap\Wz$. The following are standard results in the classical Sturm--Liouville theory.
 \begin{lemma}\label{lem:selfadj}
Under the standing hypothesis \ref{ass-s1}:
\begin{enumerate}
\item[(i)] the form \eqref{eq:secondfixedsym} on $\Wz(\II,\R^n)$ is bounded, symmetric and bilinear; the induced bounded self-adjoint operator is Fredholm of index zero;
\item[(ii)] the unbounded operator \eqref{eq:sturm} on $L^2(\II,\R^n)$ with domain $W^{2,2}\cap\Wz$ is self-adjoint and bounded below, with compact resolvent, hence has discrete spectrum and finite Morse index;
\item[(iii)] the kernels and Morse indices of the two realizations coincide.
\end{enumerate}
The same statements hold on $\Wz(\II,\R^n)\times\R$ for the block operator $I$ of \eqref{eq:block}.
\end{lemma}

Both realizations thus have the same (finite) Morse index
\begin{equation}\label{eq:morsefixed}
\mm_T(x):=\text{number of negative eigenvalues of }d^2\Lc^T(x),
\end{equation}
and the same kernel. Namely, $\xi\in\ker d^2\Lc^T(x)$ if and only if $\xi$ solves the \emph{Jacobi equation}
\begin{equation}\label{eq:jacobi}
-\frac{d}{dt}\Big[\frac1T P\dot\xi+Q\xi\Big]+Q^\top\dot\xi+TR\xi=0,\qquad \xi(0)=0=\xi(1).
\end{equation}
A nontrivial solution of \eqref{eq:jacobi} is a \emph{Jacobi field}; a point $t_*\in(0,1)$ is \emph{conjugate} to $0$ along $x$ if there is a nonzero Jacobi field on $[0,t_*]$ vanishing at the two ends.  The classical Morse index theorem (see, e.g., \cite{musso_morse_2005-1,deng_multi-dimensional_2010}) states that $\mm_T(x)$ equals the number of conjugate points in $(0,1)$ counted with multiplicity. 

\subsubsection*{The free-time Hessian}\label{subsec:freeHessian}
We now compute the second variation of the free-time functional $\Lc$ on $\Wc=\Wpm(\II,\R^n)\times\R^+$.

\begin{proposition}[Free-time Hessian]\label{prop:freeHessian}
Let $(x,T)$ be a free-time critical point of $\Lc$.  For $(\xi,b),(\eta,d)\in\Wz(\II,\R^n)\times\R$,
\begin{equation}\label{eq:freeHessian}
\begin{aligned}
d^2\Lc(x,T)[(\xi,b),(\eta,d)]
=\;&d^2\Lc^T(x)[\xi,\eta]\\
&+b\int_0^1\Big[\big\langle\bar L(t)-\tfrac1T Q(t)\dot x,\ \eta\big\rangle-\frac{1}{T^2}\angles{P(t)\dot x}{\dot\eta}\Big]dt\\
&+d\int_0^1\Big[\big\langle\bar L(t)-\tfrac1T Q(t)\dot x,\ \xi\big\rangle-\frac{1}{T^2}\angles{P(t)\dot x}{\dot\xi}\Big]dt\\
&+bd\,\frac{1}{T^3}\int_0^1\angles{P(t)\dot x}{\dot x}\,dt.
\end{aligned}
\end{equation}
The two mixed terms share the \emph{same} linear functional
\begin{equation}\label{eq:beta}
\beta(\eta):=\int_0^1\Big[\big\langle\bar L-\tfrac1T Q\dot x,\ \eta\big\rangle-\frac{1}{T^2}\angles{P\dot x}{\dot\eta}\Big]dt.
\end{equation}
\end{proposition}

\begin{proof}
Write $D(t)=\partial_vL(x,\dot x/T)$.  The first variation is
\[
d\Lc(x,T)[(\xi,b)]=T\int_0^1\big[\angles{\bar L}{\xi}+T^{-1}\angles{D}{\dot\xi}\big]dt
+b\int_0^1\big[L-T^{-1}\angles{D}{\dot x}\big]dt.
\]
Differentiating the $\xi$-part in $T$: since the explicit factors $\bar L, D$  depend on $T$, this gives
\[
\int_0^1\big[\angles{\bar L}{\xi}-T^{-1}\angles{Q\dot x}{\xi}-T^{-2}\angles{P\dot x}{\dot\xi}\big]dt.
\]
  This is the $d$-mixed term as a linear functional of $\xi$. By Schwarz symmetry the $b$-mixed term is the same functional of $\eta$.  
  
  Differentiating the $b$-part in $T$: the integrand $L-T^{-1}\angles{D}{\dot x}=-E(x,\dot x/T)$ satisfies
\[
\frac{\partial}{\partial T}\big[L-T^{-1}\angles{D}{\dot x}\big]
=-T^{-2}\angles{D}{\dot x}+T^{-2}\angles{D}{\dot x}+T^{-3}\angles{P\dot x}{\dot x}
=T^{-3}\angles{P\dot x}{\dot x},
\]
giving the $bd$ term. 
\end{proof}
 For the OM Lagrangian \eqref{eq:L}, where $P=I_n$, $Q=0$, $\bar L=-\nabla U(x)$, integration by parts and the Euler--Lagrange equation $\ddot x=-T^2\nabla U(x)$ give
\begin{equation}\label{eq:betaOM}
	\beta(\eta)=-2\int_0^1\angles{\nabla U(x(t))}{\eta(t)}\,dt,
	\qquad
	C_0:=\frac{1}{T^3}\int_0^1|\dot x|^2dt>0.
\end{equation}

Denote by $\tilde\beta\in\Wz(\II,\R^n)$ the Riesz representative of $\beta$ with respect to the Hilbert inner product of $\Wz$.  Then the Hessian $d^2\Lc(x,T)$ is represented by the block operator
\begin{equation}\label{eq:block}
I=\begin{pmatrix}A&\tilde\beta\\ \tilde\beta^{\,*}&C_0\end{pmatrix}
\quad\text{on }\Wz(\II,\R^n)\times\R,
\end{equation}
i.e.\ $d^2\Lc(x,T)[(\xi,b),(\eta,d)]=\langle I(\xi,b),(\eta,d)\rangle$. Morse index of $(x,T)$ is defined by 
\begin{equation}\label{eq:morse}
	\mm(x,T):=\text{number of negative eigenvalues of }d^2\Lc(x,T).
\end{equation}

\subsection{Spectral flow for paths of self-adjoint Fredholm operators}\label{subsec:sfprelim}

Let $\mathcal{CF}^{sa}(E)$ denote the space of self-adjoint Fredholm operators on a Hilbert space $E$.  For a $C^1$ path $w=\{w_\lambda\}_{\lambda\in[0,1]}\subset\mathcal{CF}^{sa}(E)$, a \emph{crossing instant} is a $\lambda_*$ with $\ker w_{\lambda_*}\neq0$; the \emph{crossing form} is the quadratic form
\begin{equation}\label{eq:crossingform}
\Gamma(w,\lambda_*):\ker w_{\lambda_*}\to\R,\qquad
\Gamma(w,\lambda_*)[u]=\big\langle\tfrac{d}{d\lambda}\big|_{\lambda=\lambda_*}w_\lambda\,u,\ u\big\rangle_E.
\end{equation}
A crossing is \emph{regular} if $\Gamma(w,\lambda_*)$ is nondegenerate.  The \emph{spectral flow} $\sfl(w_\lambda,\lambda\in[0,1])$ is the net count of eigenvalues crossing from negative to positive as $\lambda$ increases; for a path with only regular crossings it is computed by
\begin{equation}\label{eq:sfformula}
\sfl(w_\lambda,\lambda\in[0,1])
=-\mm(\Gamma(w,0))+\sum_{0<\lambda_*<1}\sign\Gamma(w,\lambda_*)+\mpplus(\Gamma(w,1)),
\end{equation}
where $\sign(\cdot)=\mpplus(\cdot)-\mm(\cdot)$; see, e.g. \cite[Section 2.2]{hu_index_2017}.  In particular, for a path with nondegenerate endpoints,
\begin{equation}\label{eq:sfmorse}
\sfl(w_\lambda,\lambda\in[0,1])=\mm(w_0)-\mm(w_1).
\end{equation}
Spectral flow is invariant under homotopies of paths with fixed endpoints, and additive under concatenation.  For a continuous two-parameter family $w:[a,b]\times[c,d]\to\mathcal{CF}^{sa}(E)$ one has the rectangle formula
\begin{equation}\label{eq:rectangle}
\sfl(w_{r,d},r\in[a,b])
=\sfl(w_{b,s},s\in[c,d])+\sfl(w_{r,c},r\in[a,b])-\sfl(w_{a,s},s\in[c,d]).
\end{equation}

\subsection{Regularization and index formulas}\label{subsec:regularization}

The Hessians $d^2\Lc^T(x)$ and $d^2\Lc(x,T)$ may be degenerate at the point of interest.  Following \cite{portaluri_linear_2022}, for $r\geq0$ we introduce the regularized forms
\begin{align}
\mathcal{I}^T_r[\xi,\eta]&=d^2\Lc^T(x)[\xi,\eta]+\frac rT\int_0^1\angles{P\xi}{\eta}dt,\label{eq:regfixed}\\
\mathcal{I}_r[(\xi,b),(\eta,d)]&=d^2\Lc(x,T)[(\xi,b),(\eta,d)]
+\frac rT\int_0^1\angles{P\xi}{\eta}dt\notag\\
&\qquad+\frac{r}{T^3}\Big(\int_0^1\angles{P\dot x}{\dot x}dt\Big)bd.\label{eq:regfree}
\end{align}
Let $A(r)$, $I(r)$ be the self-adjoint operators representing \eqref{eq:regfixed}, \eqref{eq:regfree} on $\Wz(\II,\R^n)$ and $\Wz(\II,\R^n)\times\R$, respectively; then
\begin{equation}\label{eq:Irblock}
I(r)=\begin{pmatrix}A(r)&\tilde\beta\\ \tilde\beta^{\,*}&(1+r)C_0\end{pmatrix},\qquad
A(r)=A+\frac rT \tilde{P},
\end{equation}
where $\tilde{P}$ denotes the operator representing the form $ \int_0^1\langle P\xi,\eta\rangle\,dt$.

\begin{proposition}[cf.\ {\cite{portaluri_linear_2022}}]\label{prop:nondegen}
There exists $r_0>0$ such that $I(r)$ and $A(r)$ are nondegenerate (hence invertible) for every $r\geq r_0$; in fact both are positive definite for $r$ large enough.  Consequently the paths $r\mapsto I(r)$, $r\mapsto A(r)$ have only regular crossings, all with positive-definite crossing forms.
\end{proposition}

\begin{proof}
	The result was proved in \cite[Lemma 3.6]{portaluri_linear_2022} by contradiction. Here we provide a direct proof.
	
Under \ref{ass-s1} the forms \eqref{eq:regfixed}--\eqref{eq:regfree} are continuous symmetric bilinear forms whose associated operators are compact perturbations of a positive multiple of the identity plus a coercive second-order part. They are thus Fredholm of index zero with discrete spectrum bounded below. We denote by
$a_r[\xi]:=\mathcal I^T_r[\xi,\xi]=a[\xi]+\frac rT\int\langle P\xi,\xi\rangle dt$. By G\aa rding's inequality $a[\xi]\ge c_G\|\xi\|^2_{W^{1,2}}-C\|\xi\|^2_{L^2}$ for some $c_G>0$, so for $r$ large,
\begin{align}
	a_r[\xi]\;\ge\;c_G\|\xi\|^2_{W^{1,2}}+\big(\tfrac{r}{T}-C\big)\|\xi\|^2_{L^2}\;\ge\;c_G\|\xi\|^2_{W^{1,2}},
\end{align} 
Therefore, $A(r)$ is positive definite. 

Since
\begin{align}
	\langle\tilde\beta,A(r)^{-1}\tilde\beta\rangle=\sup_\eta\big\{2\langle\tilde\beta,\eta\rangle-a_r[\eta]\big\},
\end{align}
there holds
\begin{align}
	\langle\tilde\beta,A(r)^{-1}\tilde\beta\rangle\le\sup_\eta\big\{2\|\tilde\beta\|_{(W^{1,2})^*}\|\eta\|_{W^{1,2}}-\tfrac{c_G}{2}\|\eta\|^2_{W^{1,2}}\big\}=\tfrac{2}{c_G}\|\tilde\beta\|^2_{(W^{1,2})^*},	
\end{align}
which is bounded above. 
Then the Schur complement 
 $(1+r)C_0-\angles{\tilde\beta}{A(r)^{-1}\tilde\beta}$ is positive for all large $r$, this yields that $I(r)$ is positive; see \eqref{eq:congruent}. By \eqref{eq:crossingform}, the crossing forms are $\Gamma(I(\cdot),r)=\frac{\tilde{P}}{T}\oplus C_0>0$. 
\end{proof}

\begin{proposition}\label{prop:sfindex}
Let $(x,T)$ be a free-time critical point of $\Lc$ and $r_0$ as in Proposition~\ref{prop:nondegen}.  Then
\begin{equation}\label{eq:sfindex}
\sfl(I(r),r\in[0,r_0])=\mm(x,T),\qquad
\sfl(A(r),r\in[0,r_0])=\mm_T(x).
\end{equation}
\end{proposition}

\begin{proof}
By a small $C^1$-perturbation of the path $\{I(r)\}$, fixed at $r=0,r_0$, we may assume all crossings are regular and that the perturbation is increasing at $r=0$, so that a possible crossing at $r=0$ has positive-definite crossing form and contributes $-\mm(\Gamma(I,0))=0$.  Every crossing is then interior with positive-definite crossing form $\frac{\tilde{P}}{T}\oplus C_0$, hence contributes $+\dim\ker I(r_*)$: the spectral flow counts the total number of negative eigenvalues of $I(0)$ that are removed as $r$ increases.  Since $I(r_0)$ is positive definite, this number equals $\mm(I(0))=\mm(x,T)$.  The argument for $A(r)$ is identical.  See  \cite[Proposition~3.5]{portaluri_linear_2021} and \cite[Proposition~3.5]{portaluri_linear_2022} for the analogous statements for periodic orbits.
\end{proof}

\section{The \texorpdfstring{$\{0,1\}$}{(0,1)}-index theorem}\label{sec:01}

This section contains one of the main result of the paper: the Morse index of a free-time critical point differs from the fixed-time Morse index by a correction $\nu\in\{0,1\}$ which is decided by a single scalar quantity---the curvature $S''(T)$ of the reduced action in the duration direction. This result is motivated by \cite{merry_paternain_2011}, where they established a relationship between the Morse indices of free-time (free-period) and fixed-time action functionals, with a $\{0,1\}$ correction determined by whether the periods increase or decrease along an energy orbit cylinder.

Let $(x,T)$ be a free-time critical point of $\Lc$.  Assume 
\begin{enumerate}[label=\textup{(A\arabic*)},series=Hseries]
	\item\label{eq:nondeg} the fixed-time second-variation operator $A$, defined by \eqref{eq:sturm},
	is invertible.
\end{enumerate}
Then the implicit function theorem applied to the smooth map
\[
\Wpm(\II,\R^n)\times\R^+\ni(\tilde{x},\tau)\longmapsto d_x\Lc(\tilde{x},\tau)\in(\Wz)^*,
\]
whose $x$-derivative at $(x,T)$ is $A$, yields $\delta>0$ and a unique $C^1$ branch
\begin{equation}\label{eq:branch}
\tau\mapsto \tilde{x}(\tau)\in\Wpm(\II,\R^n) 
\end{equation}
such that
\begin{align}\label{eq:branch1}
	\tilde{x}(T)=x,\qquad d_x\Lc(\tilde{x}(\tau),\tau)=0\ \ \forall\,\tau\in(T-\delta,T+\delta).
\end{align}
The \emph{reduced action} is the scalar function
\begin{equation}\label{eq:reduced}
S(\tau):=\Lc(\tilde{x}(\tau),\tau),\qquad \tau\in(T-\delta,T+\delta).
\end{equation}
Note that the $x$-derivative vanishes along the branch,
\begin{equation}\label{eq:envelope1}
S'(\tau)=\partial_{T}\Lc(\tilde{x}(\tau),\tau)=-E(\tau),
\end{equation}
where $E(\tau)$ is the constant energy of the orbit $\tilde{x}(\tau)$ (see Proposition~\ref{prop:firstvar}); in particular $S'(T)=-E(T)=0$, showing that $(x,T)$ is a free-time critical point.

\begin{theorem}[The $\{0,1\}$-index]\label{thm:B}
Let $(x,T)$ be a free-time critical point of $\Lc$, nondegenerate in the sense that \ref{eq:nondeg} holds and $S''(T)\neq0$.  Then
\begin{equation}\label{eq:01index}
\mm(x,T)=\mm_T(x)+\nu,\qquad
\nu=\indic_{\{S''(T)<0\}}\in\{0,1\}.
\end{equation}
Moreover, with $\beta$ and $C_0$ from \eqref{eq:beta}--\eqref{eq:betaOM},
\begin{equation}\label{eq:schur}
S''(T)=C_0-\angles{\beta}{A^{-1}\beta},
\end{equation}
where the pairing is equivalently read as $\angles{\tilde\beta}{A^{-1}\tilde\beta}_{W^{1,2}}$ or as $\angles{\beta}{A^{-1}\beta}_{L^2}$ with $A$ the Sturm--Jacobi operator on $L^2$; hence
\begin{equation}\label{eq:nschur}
\nu=\indic_{\{C_0<\angles{\beta}{A^{-1}\beta}\}}.
\end{equation}
\end{theorem}

\begin{proof}
\emph{Step 1: Sylvester inertia for the block operator.}  On the Hilbert space $\Wz(\II,\R^n)\times\R$ the Hessian is represented by the bounded self-adjoint block operator $I=\begin{pmatrix}A&\tilde\beta\\ \tilde\beta^{\,*}&C_0\end{pmatrix}$ of \eqref{eq:block}, with $A$ invertible by \ref{eq:nondeg}.  The operator $M:=\begin{pmatrix}\mathrm{id}&0\\ \tilde\beta^{\,*}A^{-1}&1\end{pmatrix}$ is bounded and boundedly invertible (with inverse $\begin{pmatrix}\mathrm{id}&0\\ -\tilde\beta^{\,*}A^{-1}&1\end{pmatrix}$), and
\begin{equation}\label{eq:congruent}
I=M\begin{pmatrix}A&0\\ 0&C_0-\angles{\tilde\beta}{A^{-1}\tilde\beta}\end{pmatrix}M^*.
\end{equation}
%By Sylvester's law of inertia for bounded self-adjoint operators on a Hilbert space 
Note that congruence by an invertible operator preserves the dimension of the maximal negative subspace, \eqref{eq:congruent} gives
\begin{equation}\label{eq:sylvester}
\begin{aligned}
\mm(x,T)=\mm(I)&=\mm(A)+\mm\big(C_0-\angles{\tilde\beta}{A^{-1}\tilde\beta}\big)\\
&=\mm_T(x)+\indic_{\{C_0-\angles{\beta}{A^{-1}\beta}<0\}}.
\end{aligned}
\end{equation}
where the Morse index of $A$ equals $\mm_T(x)$ by the definition \eqref{eq:morsefixed}; the Schur complement is one-dimensional and contributes one negative direction when it is negative.

\emph{Step 2: the envelope identity \eqref{eq:schur}.} We differentiate \eqref{eq:envelope1} at $\tau=T$.  With $\xi_T:=\frac{d}{d\tau}\big|_{\tau=T}\tilde{x}(\tau)\in\Wz(\II,\R^n)$,
\begin{equation}\label{eq:envelope2}
S''(T)=\frac{d}{d\tau}\Big|_{\tau=T}\partial_T\Lc(\tilde{x}(\tau),\tau)
=\partial_T^2\Lc(x,T)+\partial_T d_x\Lc(x,T)[\xi_T].
\end{equation}
The pure time derivative was computed in Proposition~\ref{prop:freeHessian}:
\begin{equation}\label{eq:puretime}
\partial_T^2\Lc(x,T)=\frac{1}{T^3}\int_0^1\angles{P\dot x}{\dot x}dt=:C_0,
\end{equation}
and the mixed term, as has been computed in the proof of Proposition~\ref{prop:freeHessian}, is
a linear functional
\begin{equation}\label{eq:mixedtime}
\partial_T d_x\Lc(x,T)[\xi_T]=\beta(\xi_T).
\end{equation}
Fix any $\eta\in\Wz$ and differentiate the branch identity $d_x\Lc(\tilde{x}(\tau),\tau)[\eta]=0$ at $\tau=T$, we get
\begin{equation}\label{eq:branchdiff}
d^2\Lc^T(x)[\xi_T,\eta]+\beta(\eta)=0\qquad\forall\,\eta\in\Wz,
\end{equation}
i.e.\ $A\xi_T=-\tilde\beta$, whence
\begin{equation}\label{eq:xiT}
\xi_T=-A^{-1}\tilde\beta.
\end{equation}
Substituting \eqref{eq:puretime}, \eqref{eq:mixedtime} and \eqref{eq:xiT} into \eqref{eq:envelope2},
\[
S''(T)=C_0+\beta(\xi_T)=C_0-\beta(A^{-1}\tilde\beta)=C_0-\angles{\beta}{A^{-1}\beta},
\]
which is \eqref{eq:schur}.  Here we use the identity
\begin{align*}
	\angles{\tilde\beta}{A^{-1}\tilde\beta}_{W^{1,2}}=\beta(-\xi_T)=\angles{\beta}{A^{-1}\beta}_{L^2}.
\end{align*}

\emph{Step 3: conclusion.}  Combining \eqref{eq:sylvester} with \eqref{eq:schur} gives \eqref{eq:01index} and \eqref{eq:nschur}.  The nondegeneracy $S''(T)\neq0$ says precisely that the scalar Schur complement is nonzero, i.e.\ that $I$ is invertible.
\end{proof}

\begin{remark} We would like to compare the result \eqref{eq:01index} with the formula in the periodic (free-period) setting \cite[Theorem 1.3]{merry_paternain_2011}, where they give
	\begin{align*}
		0\leq i(x,T)-i_T(x)\leq 1,
	\end{align*} 
	with the $+1$ determined by the sign of $T'(E)$.  Recall from \eqref{eq:envelope1} that $S''(T)=-E'(T)$. If locally one writes $T=T(E)$, then $E'(T)=\frac{1}{T'(E)} $
and consequently, 
	\begin{align}
S''(T)<0 \Leftrightarrow T'(E)>0.
	\end{align}	
	The two formulas coincide only if we ignore the difference of boundary conditions in the variational setting.
	We also refer interested readers to \cite[Proposition~2.3]{abbondandolo_existence_2015} and \cite[Corollary 3.16]{portaluri_linear_2022} for  applications of the formula in stability analysis of periodic orbits. 
\end{remark}

\begin{remark}\label{rmk:thmB}
(a) The free-time correction $\nu$ is a generic duration-direction effect. In the OM problem, its value and its noise-induced changes acquire an additional \(\sigma\)-dependence through the divergence term \(-\sigma\Delta V\). Section~\ref{subsec:dwfree} exhibits a model in which $\nu=1$ is realized with $\mm_T=0$. For the $\sigma$-independent Freidlin--Wentzell action the same formalism applies with $U$ replaced by $-\frac12|\nabla V|^2$.

(b) The hypothesis \ref{eq:nondeg} says that $T$ is not a conjugate time of $x$; the hypothesis $S''(T)\neq0$ says that the duration direction is not in the kernel.  Both are generic.

(c) Since $\mm(x,T),\mm_T(x)\geq0$, a local minimizer $(x,T)$ of $\Lc$ necessarily has $\mm_T(x)=0$ and $\nu=0$.
\end{remark}

\section{Spectral flow and the bifurcation criterion}\label{sec:sf}

In this section we detect bifurcation points of the transition-path problem with respect to the noise intensity $\sigma$ by spectral flow, and prove Theorem~C.  We work with the $\sigma$-dependent Lagrangian
\begin{align}
	L(\sigma,x,v)=&\,L_1(x,v)+\sigma L_2(x,v),\label{eq:Lsplit}\\
	L_1(x,v)=&\,\frac12|v|^2+\frac12|\nabla V(x)|^2,\quad
	L_2(x,v)=-\Delta V(x),\notag
\end{align}
and $\Lc(\sigma,x,T)=\Lc_1(x,T)+\sigma\Lc_2(x,T)$.
A $C^1$-branch $(x_\sigma,T_\sigma)$, $\sigma\in[\sigma_0,\sigma_1]$ are free-time critical points of Lagrangian $\Lc(\sigma,\cdot,\cdot)$ under the standing hypothesis (S2), or fixed-time critical points of $\Lc^\tau(\sigma)$ in the truncated setting; see Remark~\ref{rmk:fixedtimebif}.

Variational bifurcation and the Fitzpatrick--Pejsachowicz--Recht theorem are introduced as follows.

\begin{definition}\label{defi:bifpoint}
Let $\Pc:I\times X\to Y$ be a $C^l$ map ($l\geq2$) between Banach spaces, with $\Pc(\mu,0)=0$ for all $\mu$ in an interval $I$ (the \emph{trivial branch}).  A parameter value $\mu_*\in I$ is a \emph{bifurcation point} if every neighborhood of $(\mu_*,0)$ in $I\times X$ contains zeros of $\Pc$ not lying on the trivial branch.
\end{definition}

\begin{theorem}[Fitzpatrick--Pejsachowicz--Recht {\cite[Theorem~1]{fitzpatrick_spectral_1999}}]\label{thm:fpr}
Let $\mathcal{U}$ be a neighborhood of $I\times\{0\}$ in $\R\times H$, $H$ a Hilbert space, and $\psi:\mathcal{U}\to\R$ a $C^2$ function such that $0$ is a critical point of $\psi_\lambda:=\psi(\lambda,\cdot)$ for every $\lambda\in I=[a,b]$.  Assume that the Hessian $L_\lambda$ of $\psi_\lambda$ at $0$ is Fredholm for every $\lambda$ and that $L_a$, $L_b$ are nonsingular.  If
\[
\sfl(L_\lambda,\lambda\in[a,b])\neq0,
\]
then there exists a bifurcation point $\lambda_*\in[a,b]$ for the equation $\nabla\psi_\lambda(x)=0$; more precisely, every neighborhood of $I\times\{0\}$ contains points $(\lambda,x)$ with $x\neq0$ a critical point of $\psi_\lambda$.
\end{theorem}

\begin{remark}[Trivial-branch reduction]\label{rmk:trivialbranch}
Our functional $\Lc(\sigma,\cdot,\cdot)$ has the branch $(x_\sigma,T_\sigma)$ of critical points, not the trivial branch.  The $C^1$ change of variables
\begin{equation}\label{eq:trivial}
(\sigma,x,T)\longmapsto \big(\sigma,\ (x,T)-(x_\sigma,T_\sigma)\big)
\end{equation}
reduces the branch to the trivial branch; the Hessians at the trivial branch are exactly the operators $I_\sigma(0)$ of \eqref{eq:Irblock} (with $r=0$), which are Fredholm of index zero by (S1).  The bifurcation statements below always refer to the equation
\begin{equation}\label{eq:P}
\Pc(\sigma,x,T):=\nabla\Lc(\sigma,x,T)=0
\end{equation}
after the reduction \eqref{eq:trivial}.
\end{remark}

\subsection{Hamiltonian form of the Jacobi equation}\label{subsec:hamiltonian}

Let $P_\sigma,Q_\sigma,R_\sigma$ be the matrix fields \eqref{eq:pqr} along $(x_\sigma,\dot x_\sigma/T_\sigma)$.  For $\xi\in\Wz(\II,\R^n)$ set
\begin{equation}\label{eq:momentum}
y(t)=\frac{1}{sT_\sigma}P_\sigma(st)\dot\xi(t)+Q_\sigma(st)\xi(t),\qquad u(t)=(y(t),\xi(t))\in\R^{2n}.
\end{equation}
The Dirichlet conditions $\xi(0)=0=\xi(1)$ read $(u(0),u(1))\in L_D\times L_D$, where
\begin{equation}\label{eq:Lambda}
L_D:=\R^n\times\{0\},
\end{equation}
and $L_D\times L_D$ is a Lagrangian subspace of $\R^{2n}\times\R^{2n}$ for the symplectic form $\omega\oplus(-\omega)$; $\omega$ the standard form with matrix $J=\begin{pmatrix}0&-I_n\\ I_n&0\end{pmatrix}$.

  The $s$-scaled second variation---the second variation of $\Lc^{sT_\sigma}$ at the restricted orbit $x_\sigma(s\,\cdot)$, pulled back to $\II$---is, by \eqref{eq:secondfixedsym},
\begin{equation}\label{eq:sscaled}
\begin{aligned}
d^2\Lc^{sT_\sigma}(x_\sigma(s\cdot))[\xi,\eta]
=\int_0^1\Big[\frac{1}{sT_\sigma}&\angles{P_\sigma(st)\dot\xi}{\dot\eta}
+\angles{Q_\sigma(st)\xi}{\dot\eta}\\
&+\angles{Q_\sigma^\top(st)\dot\xi}{\eta}
+sT_\sigma\angles{R_\sigma(st)\xi}{\eta}\Big]dt,
\end{aligned}
\end{equation}
whose associated operator is the $s$-scaled Sturm--Jacobi operator
\[
A_\sigma(s)\xi=-\frac{d}{dt}\Big[\frac{1}{sT_\sigma}P_\sigma(st)\dot\xi+Q_\sigma(st)\xi\Big]
+Q_\sigma^\top(st)\dot\xi+sT_\sigma R_\sigma(st)\xi.
\]
With the momentum \eqref{eq:momentum} one has $\dot\xi=sT_\sigma P_\sigma^{-1}(st)(y-Q_\sigma(st)\xi)$, and the equation $A_\sigma(s)\xi=0$ reads
\[
\dot y=Q_\sigma^\top(st)\dot\xi+sT_\sigma R_\sigma(st)\xi
=sT_\sigma\Big[Q_\sigma^\top P_\sigma^{-1}y-\big(Q_\sigma^\top P_\sigma^{-1}Q_\sigma-R_\sigma\big)\xi\Big](st),
\]
i.e.\ precisely the linear Hamiltonian system
\begin{equation}\label{eq:hamjacobi}
\dot u(t)=JB_{\sigma,s}(t)u(t),\qquad (u(0),u(1))\in L_D\times L_D,
\end{equation}
with the symmetric matrix family
\begin{equation}\label{eq:B}
B_{\sigma,s}(t)=sT_\sigma
\begin{pmatrix}
P^{-1} & -P^{-1}Q\\
-Q^\top P^{-1} & Q^\top P^{-1}Q-R
\end{pmatrix}(st)
=:sB_\sigma(st),
\end{equation}
where $B_\sigma(\cdot)$ denotes the bracketed matrix multiplied by $T_\sigma$.  We denote by
\begin{equation}\label{eq:Aop}
\Ac(\sigma,s):=-J\frac{d}{dt}-B_{\sigma,s}(t):\ W^{1,2}_D\longrightarrow L^2(\II,\R^{2n}),
\end{equation}
with $W^{1,2}_D:=\{u\in W^{1,2}(\II,\R^{2n}):(u(0),u(1))\in L_D\times L_D\}$, the associated self-adjoint Fredholm operator, continuous up to $s=0$ with $\Ac(\sigma,0)=-J\frac{d}{dt}$.  Kernels correspond to Jacobi fields:
\begin{equation}\label{eq:kernelcorresp}
\ker\Ac(\sigma,s)\neq0\iff s\text{ is a conjugate instant of }(x_\sigma,T_\sigma),
\end{equation}
and $\dim\ker\Ac(\sigma,s)$ equals the multiplicity of the conjugate instant.  Since the path $\sigma\mapsto\Ac(\sigma,0)=-J\frac{d}{dt}$ is constant,
\begin{equation}\label{eq:sfA0}
\sfl(\Ac(\sigma,0),\sigma\in[\sigma_0,\sigma_1])=0.
\end{equation}

\subsection{Crossing forms and the index identity}\label{subsec:crossing}

For a crossing instant $(\sigma,s)$, write $K(\sigma,s):=\ker\Ac(\sigma,s)$.  Since $\partial_\sigma\Ac=-\partial_\sigma B_{\sigma,s}$ and $\partial_s\Ac=-\partial_s B_{\sigma,s}$, the crossing forms \eqref{eq:crossingform} in the two directions are
\begin{align}
\Gamma(\Ac(\sigma,s),\sigma)[u,v]&=-\angles{\partial_\sigma B_{\sigma,s}\,u}{v}_{L^2}\Big|_{K(\sigma,s)},\label{eq:crosssigma}\\
\Gamma(\Ac(\sigma,s),s)[u,v]&=-\angles{\partial_s B_{\sigma,s}\,u}{v}_{L^2}\Big|_{K(\sigma,s)}.\label{eq:crosss}
\end{align}

\begin{lemma}[$s$-crossings are regular and downward]\label{lem:scrossing}
Every crossing of the path $s\mapsto\Ac(\sigma,s)$ is regular in the $s$-direction, and
\begin{equation}\label{eq:signs}
\sign\Gamma(\Ac(\sigma,s),s)=-\dim K(\sigma,s)<0,
\end{equation}
i.e.\ $\Gamma(\Ac(\sigma,s),s)$ is negative definite on $K(\sigma,s)$.
\end{lemma}

\begin{proof}
Let $\phi_{\sigma,s}:\II\to\Spm(2n)$ be the fundamental solution of \eqref{eq:hamjacobi}, $\phi_{\sigma,s}(0)=I_{2n}$.  By the boundary representation of crossing forms (Appendix~\ref{app:crossing}, Theorem~\ref{thm:boundary}),
\[
\angles{\partial_s B_{\sigma,s}u}{v}_{L^2}
=\angles{-J\phi_{\sigma,s}^{-1}(1)\,\partial_s\phi_{\sigma,s}(1)\,u(0)}{v(0)},
\qquad u,v\in K(\sigma,s).
\]
Since $s^{-1}\Ac(\sigma,s)=-J\frac{d}{d(s\cdot)}-B_\sigma(s\cdot)$, we have $\phi_{\sigma,s}(t)=\phi_\sigma(st)$ with $\phi_\sigma'=JB_\sigma\phi_\sigma$, and
\[
\frac{d}{ds}\phi_{\sigma,s}(t)=t\,JB_\sigma(st)\phi_{\sigma,s}(t).
\]
The symplectic identity $\phi^\top J\phi=J$ gives $\phi^{-1}=-J\phi^\top J$, hence $-J\phi^{-1}J=\phi^\top$ and
\[
-J\phi_{\sigma,s}^{-1}(1)\,\partial_s\phi_{\sigma,s}(1)
=-J\phi_{\sigma,s}^{-1}(1)\,JB_\sigma(s)\phi_{\sigma,s}(1)
=\phi_{\sigma,s}^\top(1)B_\sigma(s)\phi_{\sigma,s}(1).
\]
For $u\in K(\sigma,s)$ the boundary condition gives $u(1)=\phi_{\sigma,s}(1)u(0)\in\R^n\times\{0\}$, so only the positive definite $(y,y)$-block $T_\sigma P_\sigma^{-1}(s)$ of $B_\sigma(s)$ contributes:
\begin{align*}
\angles{\partial_s B_{\sigma,s}u}{u}_{L^2}
=& \angles{B_\sigma(s)u(1)}{u(1)}\\
=& T_\sigma\angles{P_\sigma^{-1}(s)y_u(1)}{y_u(1)}\geq \frac{T_\sigma}{\|P_\sigma(s)\|}\,|y_u(1)|^2\geq0.	
\end{align*}
If the right-hand side vanished, then $y_u(1)=0$; together with $\xi_u(1)=0$ this gives $u(1)=0$, hence $u\equiv0$ by uniqueness for \eqref{eq:hamjacobi}.  Thus $\angles{\partial_sB_{\sigma,s}\cdot}{\cdot}$ is positive definite on $K(\sigma,s)$, and \eqref{eq:signs} follows from the minus sign in \eqref{eq:crosss}.
\end{proof}
The spectral flow of $\mathcal{A}(\cdot,\cdot)$ along the edge $\{\sigma\}\times [0,1]$ is as follows.
\begin{proposition}[Index identity along the $s$-path]\label{prop:indexidentity}
For each $\sigma$, if $s=1$ is not a crossing (i.e.\ the fixed-time problem at $(x_\sigma,T_\sigma)$ is nondegenerate), then
\begin{equation}\label{eq:indexidentity}
\sfl(\Ac(\sigma,s),s\in[0,1])=-n-\mm_{T_\sigma}(x_\sigma),
\end{equation}
where $n=\dim\R^n$ is the space dimension.  In particular the $\sigma$-dependent part is exactly $-\mm_{T_\sigma}(x_\sigma)$.
\end{proposition}

\begin{proof}
We apply the crossing-form formula \eqref{eq:sfformula}.  At $s=0$: $\Ac(\sigma,0)=-J\frac{d}{dt}$ has kernel consisting of the constant paths $u(t)\equiv(y_0,0)$, so $\dim K(\sigma,0)=n$.  Since $\partial_sB_{\sigma,s}(t)|_{s=0}=B_\sigma(0)$, the crossing form is
\[
\Gamma(\Ac(\sigma,\cdot),0)[u]=-\angles{B_\sigma(0)u}{u}_{L^2}=-T_\sigma\angles{P_\sigma^{-1}(0)y_0}{y_0},
\]
negative definite of dimension $n$; its contribution to \eqref{eq:sfformula} is $-\mm(\Gamma(0))=-n$.  At an interior crossing $s_*\in(0,1)$, Lemma~\ref{lem:scrossing} gives the contribution $\sign\Gamma=-\dim K(\sigma,s_*)$.  By the Morse index theorem for the Sturm system \eqref{eq:jacobi} (cf.\ Section~\ref{subsec:fixedsecond}), the total multiplicity of conjugate instants in $(0,1)$ equals $\mm_{T_\sigma}(x_\sigma)$.  A possible crossing at $s=1$ is excluded by hypothesis (its contribution would be $\mpplus(\Gamma(1))=0$ anyway, by Lemma~\ref{lem:scrossing}).  Summing gives \eqref{eq:indexidentity}.
\end{proof}

%\begin{remark}[Endpoint conventions]\label{rmk:endpoint}
%The constant $-n$ in \eqref{eq:indexidentity} is the contribution of the left endpoint $s=0$ under the convention \eqref{eq:sfformula}; it is \emph{independent of $\sigma$} (it depends only on the space dimension).  Since only differences over $\sigma$ enter Theorem~\ref{thm:C} below, the constant cancels out of every application; equivalently, one may renormalize the $s$-path spectral flow by $+n$.
%\end{remark}
Now we consider the spectral flow of $\mathcal{A}(\cdot,\cdot)$ along the edge $[\sigma_0,\sigma_1]\times\{1\}$.
\begin{lemma}\label{lem:sfA1}
If the fixed-time problems at $\sigma_0,\sigma_1$ are nondegenerate, then
\begin{equation}\label{eq:sfA1}
\sfl(\Ac(\sigma,1),\sigma\in[\sigma_0,\sigma_1])
=\mm_{T_{\sigma_0}}(x_{\sigma_0})-\mm_{T_{\sigma_1}}(x_{\sigma_1}).
\end{equation}
\end{lemma}

\begin{proof}
Apply the rectangle formula \eqref{eq:rectangle} to $(\sigma,s)\mapsto\Ac(\sigma,s)$ on $[\sigma_0,\sigma_1]\times[0,1]$: the boundary loop has vanishing spectral flow, and the bottom edge contributes \eqref{eq:sfA0}, so
\[
\sfl(\Ac(\sigma,1),\sigma)=\sfl(\Ac(\sigma_1,s),s)-\sfl(\Ac(\sigma_0,s),s)
=\big[-n-\mm_{T_{\sigma_1}}\big]-\big[-n-\mm_{T_{\sigma_0}}\big]
\]
by Proposition~\ref{prop:indexidentity}, which is \eqref{eq:sfA1}.
\end{proof}

\subsection{The proof of Theorem~C}\label{subsec:thmC}

Let $(x_\sigma,T_\sigma)$, $\sigma\in[\sigma_0,\sigma_1]$, be a $C^1$ branch of free-time critical points of $\Lc(\sigma,\cdot,\cdot)$, and let $I_\sigma(r)$, $A_\sigma(r)$ be the regularized Hessians \eqref{eq:Irblock} at $(x_\sigma,T_\sigma)$.  Define the scalar path
\begin{equation}\label{eq:fraC}
\fa(\sigma):=S_\sigma''(T_\sigma),
\end{equation}
the curvature of the reduced action \eqref{eq:reduced} along the branch.  For a scalar path with regular crossings (simple zeros), the spectral flow is
\begin{equation}\label{eq:sfscalar}
\sfl(\fa(\sigma),\sigma\in[\sigma_0,\sigma_1])=\mpplus(\fa(\sigma_1))-\mpplus(\fa(\sigma_0)),
\end{equation}
where $\mpplus$ of a nonzero scalar is $1$ if it is positive and $0$ if it is negative.

\begin{theorem}[Bifurcation formula]\label{thm:C}
Assume the endpoints of the full Hessian $I_{\sigma_i}(0),i=0,1$ are nondegenerate and the fixed-time problem has no conjugate point at $s=1$ (equivalently, the fixed-time second-variation operator $A_{\sigma_i}$ is nondegenerate), i.e.
\begin{equation}\label{eq:endpointnondeg}
\ker I_{\sigma_i}(0)=0\quad\text{and}\quad \ker\Ac(\sigma_i,1)=0,\qquad i=0,1,
\end{equation}
 Note that the first condition alone does not imply the second.
Then
\begin{equation}\label{eq:thmC}
\begin{aligned}
\sfl(I_\sigma(0),\sigma\in[\sigma_0,\sigma_1])
&=\mm(x_{\sigma_0},T_{\sigma_0})-\mm(x_{\sigma_1},T_{\sigma_1})\\
&=\big[\mm_{T_{\sigma_0}}(x_{\sigma_0})-\mm_{T_{\sigma_1}}(x_{\sigma_1})\big]+\big[\nu(\sigma_0)-\nu(\sigma_1)\big]\\
&=\sfl\big(\fa(\sigma),\sigma\in[\sigma_0,\sigma_1]\big)+\sfl\big(\Ac(\sigma,1),\sigma\in[\sigma_0,\sigma_1]\big),
\end{aligned}
\end{equation}
with $\fa$ as in \eqref{eq:fraC} and $\Ac$ as in \eqref{eq:Aop}. If the two crossings do not occur simultaneously, then a nonzero spectral flow of either $\mathfrak{a}$ or $\mathcal{A}$ implies the existence of a bifurcation point $\sigma_*\in[\sigma_0,\sigma_1]$ for \eqref{eq:P}.
\end{theorem}

\begin{proof}
\emph{Step 1.}  Apply the rectangle formula \eqref{eq:rectangle} to $(\sigma,r)\mapsto I_\sigma(r)$ on $[\sigma_0,\sigma_1]\times[0,r_0]$, with $r_0$ as in Proposition~\ref{prop:nondegen} (taken uniformly in $\sigma$ by compactness).  The top edge $\{(\sigma,r_0)\}$ consists of positive definite operators and contributes $0$; hence
\[
\sfl(I_\sigma(0),\sigma)=\sfl(I_{\sigma_0}(r),r)-\sfl(I_{\sigma_1}(r),r)
=\mm(x_{\sigma_0},T_{\sigma_0})-\mm(x_{\sigma_1},T_{\sigma_1}),
\]
where we used Proposition~\ref{prop:sfindex} and the endpoint nondegeneracy \eqref{eq:endpointnondeg}.  This proves the first line of \eqref{eq:thmC}.

\emph{Step 2.}  Theorem~\ref{thm:B} gives $\mm(x_\sigma,T_\sigma)=\mm_{T_\sigma}(x_\sigma)+\nu (\sigma)$ at $\sigma=\sigma_0,\sigma_1$ (both endpoints satisfy \ref{eq:nondeg} and $S_\sigma''(T_\sigma)\neq0$ by \eqref{eq:endpointnondeg}), proving the second line.

\emph{Step 3.}  For the scalar path, \eqref{eq:sfscalar} and $\nu (\sigma)=\indic_{\{\fa(\sigma)<0\}}$ (Theorem~\ref{thm:B}) give
\[
\sfl(\fa)=\mpplus(\fa(\sigma_1))-\mpplus(\fa(\sigma_0))=(1-\nu (\sigma_1))-(1-\nu (\sigma_0))=\nu (\sigma_0)-\nu (\sigma_1).
\]
Together with Lemma~\ref{lem:sfA1} this proves the third line of \eqref{eq:thmC}.

\emph{Step 4.}  Under the trivial-branch reduction \eqref{eq:trivial}, Theorem~\ref{thm:fpr} applies to $\psi(\sigma,\cdot):=\Lc(\sigma,\cdot,\cdot)$ with Hessians $I_\sigma(0)$, which are Fredholm self-adjoint and nonsingular at $\sigma_0,\sigma_1$ by \eqref{eq:endpointnondeg}.  Nonzero spectral flow then yields a bifurcation point in $[\sigma_0,\sigma_1]$.
\end{proof}

\begin{remark}\label{rmk:fixedtimebif}
In the fixed-time (truncated) setting $k_0(L;x_\pm)>0$ of Theorem~\ref{thm:A}(ii), the same argument applied to the branch $x_\sigma$ of critical points of $\Lc^\tau(\sigma)$, with the rectangle over $(\sigma,r)\mapsto A_\sigma(r)$, gives
	\begin{equation}\label{fixedtimebif}
		\begin{aligned}
	\sfl(A_\sigma(0),\sigma\in[\sigma_0,\sigma_1])
		=&\mm_\tau(x_{\sigma_0})-\mm_\tau(x_{\sigma_1})\\
		=&\sfl(\Ac(\sigma,1),\sigma\in[\sigma_0,\sigma_1]),
		\end{aligned}	
	\end{equation}
and the bifurcation statement of Theorem~\ref{thm:fpr} holds with $A_\sigma(0)$ in place of $I_\sigma(0)$. That is, nonzero spectral flow implies a bifurcation point of $\nabla\Lc^\tau(\sigma,\cdot)=0$ in the interval.
\end{remark}

A change in the Morse indices of 
$(x_{\sigma_0},T_{\sigma_0})$ or $x^\tau_{\sigma_0}$
along $\sigma\in[\sigma_0,\sigma_1]$ indicates a variation in their stability. As noted in the first line of both \eqref{eq:thmC} and \eqref{fixedtimebif}, this variation is equivalent to a non-vanishing spectral flow. However, the two formulas go further: they allow one to determine the precise bifurcation point by computing the spectral flows of $\fa$ and $\Ac$ through crossing forms. We work on this in next section.
 
%\begin{remark}[On signs]\label{rmk:signsC}
%The two terms of \eqref{eq:thmC} have distinct geometric content: $\sfl(\Ac(\sigma,1))$ counts the net change of the conjugate-point count $\mm_{T_\sigma}$ (the \emph{conjugate mechanism}), while $\sfl(\fa)$ counts the net change of the $\{0,1\}$-index $n$ (the \emph{duration mechanism}). 
%\end{remark}

\section{Local stability under noise variation}\label{sec:stab}
In this section, 
we turn to the question of when local minimality along a branch persists or is destroyed as $\sigma$ varies.  Recall that ``stability'' in this paper always means \emph{local} minimality, then $\mm=0$, along a $C^1$ branch of critical points.

\begin{definition}\label{defi:stable}
A critical point $(x_{\sigma_0},T_{\sigma_0})$ (resp.\ $x_{\sigma_0}$ in the fixed-time setting) is \emph{positively stable} if for all sufficiently small $\delta>0$ the branch elements with $\sigma\in(\sigma_0,\sigma_0+\delta)$ are local minimizers ($\mm=0$ resp.\ $\mm_{T_\sigma}=0$), and \emph{positively unstable} if every interval $(\sigma_0,\sigma_0+\delta)$ contains parameters for which local minimality fails.
\end{definition}

\subsection{The crossing curve}\label{subsec:crossingcurve}

We use the following finite-dimensional reduction of the crossing set, in the spirit of \cite{cox_hamiltonian_2023}.

\begin{lemma}[Scalar reduction of the crossing set]\label{lem:scalarred}
Let $(\sigma_*,s_*)$ be a crossing instant of $\Ac$ with 1-dimensional kernel: $\dim K(\sigma_*,s_*)=1$.  Then there is a neighborhood of $(\sigma_*,s_*)$ and a $C^1$ scalar function $M(\sigma,s)$ such that
\[
\ker\Ac(\sigma,s)\neq\{0\}\iff M(\sigma,s)=0,
\]
and for $z\in K(\sigma_*,s_*)$,
\begin{equation}\label{eq:derM}
\partial_\sigma M(\sigma_*,s_*)=\Gamma(\Ac,\sigma)[z],\qquad
\partial_s M(\sigma_*,s_*)=\Gamma(\Ac,s)[z].
\end{equation}
\end{lemma}

\begin{proof}
Let $P_K$ be the $L^2$-orthogonal projection onto $K:=K(\sigma_*,s_*)$. Then $u\in K(\sigma,s)$ if and only if
\begin{align}
	(I-P_K)\mathcal{A}(\sigma,s)u=0,\label{proji-k}\\
	P_K\mathcal{A}(\sigma,s)u=0.\label{projk}
\end{align}
Denote by $X_0:=(K^\perp)\cap W^{1,2}_D\cap W^{2,2}$. Then each $u\in W^{1,2}_D\cap W^{2,2}$ has the form $u=az+x$, where $a\in\mathbb{R},\,z\in K$ and $x\in X_0$.

We write $\Ac(\sigma,s)$ in the splitting $W^{1,2}_D=K\oplus X_0$.  The operator $T(\sigma,s):=(I-P_K)\Ac(\sigma,s)|_{X_0}:X_0\to\ran\Ac(\sigma_*,s_*)$ is invertible at $(\sigma_*,s_*)$, and therefore invertible in its neighborhood.

Now we show that $u$ is solvable by $z$ if $u\in K(\sigma,s)$. This is essentially the finite dimensional reduction theorem.
The equation \eqref{proji-k} gives
\begin{align*}
(I-P_K)\mathcal{A}(\sigma,s)x&=-(I-P_K)\mathcal{A}(\sigma,s)\,az,
\end{align*}
and hence $x=-T^{-1}(\sigma,s)(I-P_K)\mathcal{A}(\sigma,s)az$. Denote by
$\Kc(\sigma,s):=-T(\sigma,s)^{-1}(I-P_K)\Ac(\sigma,s)|_{K}$, then
\begin{align*}
	u=az+x=(I+\Kc(\sigma,s))(az).
\end{align*} We conclude that $u$ is a function of $z$ and is uniquely determined by $a\in\mathbb{R}$.
Now the remaining equation $P_K\Ac(\sigma,s)u=0$ gives
\begin{align*}
		P_K\mathcal{A}(\sigma,s)(I+\Kc(\sigma,s))(az)=0,
\end{align*}
or equivalently, $\angles{\mathcal{A}(\sigma,s)(I+\Kc(\sigma,s))(az)}{y}=0,\, \forall y\in K$. 

We denote by
\[
M(\sigma,s):=\angles{\Ac(\sigma,s)(I+\Kc(\sigma,s))z}{z}_{L^2},
\]
which is a scalar function and $C^1$-continuous w.r.t $(\sigma,s)$.
Since $\Ac(\sigma_*,s_*)z=0$ and $\Kc(\sigma_*,s_*)z=0$, differentiation gives $\partial_\sigma M=\angles{\partial_\sigma\Ac\,z}{z}$ and $\partial_sM=\angles{\partial_s\Ac\,z}{z}$, which are \eqref{eq:crosssigma}--\eqref{eq:crosss}.
\end{proof}

\begin{proposition}[The crossing curve and its slope]\label{prop:crossingcurve}
At a crossing instant $(\sigma_*,s_*)$ with $\dim K(\sigma_*,s_*)=1$, the crossing set $\{(\sigma,s):\ker\Ac(\sigma,s)\neq0\}$ is locally the graph of a $C^1$ function $s(\sigma)$ with $s(\sigma_*)=s_*$, and
\begin{equation}\label{eq:sprime}
s'(\sigma)=-\frac{\Gamma(\Ac(\sigma,s(\sigma)),\sigma)[u]}{\Gamma(\Ac(\sigma,s(\sigma)),s)[u]},
\qquad u\in K(\sigma,s(\sigma)),\ \|u\|=1.
\end{equation}
\end{proposition}

\begin{proof}
By Lemma~\ref{lem:scrossing}, $\Gamma(\Ac,s)[u]<0$ at the crossing, so $\partial_sM(\sigma_*,s_*)\neq0$ by \eqref{eq:derM}.  The implicit function theorem applied to $M(\sigma,s)=0$ yields the $C^1$ graph $s(\sigma)$. Differentiating the identity $M(\sigma,s(\sigma))=0$ with respect to $\sigma$ gives
\[
s'(\sigma)=-\frac{\partial_\sigma M}{\partial_s M}=-\frac{\Gamma(\Ac,\sigma)[u]}{\Gamma(\Ac,s)[u]},
\]
which is \eqref{eq:sprime}.
\end{proof}

\subsection{Theorem~D: fixed-time local stability}\label{subsec:thmD}

\begin{theorem}[Local stability criterion]\label{thm:D}
Let $\{x_\sigma\}$ be a $C^1$ branch of critical points of the fixed-time functionals $\Lc^\tau(\sigma)$ (e.g.\ the truncated minimizers of Theorem~\ref{thm:A}\emph{(ii)}), such that $x_{\sigma_0}$ is a local minimizer ($\mm_\tau(x_{\sigma_0})=0$) and
\begin{equation}\label{eq:simplecross}
\dim\ker\Ac(\sigma_0,1)=1.
\end{equation}
Let $s(\sigma)$ be the crossing curve through $(\sigma_0,1)$ given by Proposition~\ref{prop:crossingcurve}.  Then there exists $\delta>0$ such that:
\begin{enumerate}
\item[(1)] if $s'(\sigma_0)>0$, then $x_\sigma$ remains a local minimizer for all $\sigma\in(\sigma_0,\sigma_0+\delta)$ \emph{(noise-robust)};
\item[(2)] if $s'(\sigma_0)<0$, then $x_\sigma$ has Morse index $\mm_\tau(x_\sigma)=1$ for all $\sigma\in(\sigma_0,\sigma_0+\delta)$: it acquires exactly one negative direction and local minimality is lost.
\end{enumerate}
\end{theorem}

\begin{proof}
 By \eqref{eq:simplecross} and Proposition~\ref{prop:crossingcurve}, near $(\sigma_0,1)$ the crossing set of the two-parameter family $\Ac(\sigma,s)$ is exactly the graph $\{(\sigma,s(\sigma))\}$.  

\emph{Case $s'(\sigma_0)>0$.}  The crossing of the path $\sigma\mapsto\Ac(\sigma,1)$ at the left endpoint $\sigma_0$ has crossing form
\begin{equation}\label{eq:crossrel}
\Gamma(\Ac,\sigma)[u]=-s'(\sigma_0)\,\Gamma(\Ac,s)[u]>0
\end{equation}
(by \eqref{eq:sprime} and Lemma~\ref{lem:scrossing}), hence contributes $-\mm(\Gamma(\sigma_0))=0$ to the spectral flow, and $\sfl(\Ac(\sigma,1),\sigma\in[\sigma_0,\sigma_0+\delta])=0$ by \eqref{eq:sfformula}.

\emph{Case $s'(\sigma_0)<0$.} In this case, \eqref{eq:crossrel} gives $\Gamma(\Ac,\sigma)[u]<0$: the left-endpoint crossing contributes $-\mm(\Gamma(\sigma_0))=-1$ to
\[
\sfl(\Ac(\sigma,1),\sigma\in[\sigma_0,\sigma_0+\delta])=-1
=\mm_\tau(x_{\sigma_0})-\mm_\tau(x_{\sigma_0+\delta}),
\]
so $\mm_\tau$ increases by exactly one.
\end{proof}
We remark that if $s'(\sigma_0)=0$ the situation is degenerate and a higher-order analysis would be required; see e.g. \cite[Theorem 2.9]{cox_hamiltonian_2023}.
\subsection{Free-time stability: the two mechanisms}\label{subsec:freestab}

We now consider the free-time setting \ref{ass-s2}: $k_0(L(\sigma);x_\pm)\leq0$, and a $C^1$ branch $(x_\sigma,T_\sigma)$ of free-time critical points of $\Lc(\sigma)$.  By Theorem~\ref{thm:B},
\begin{equation}\label{eq:twomech}
\mm(x_\sigma,T_\sigma)=\mm_{T_\sigma}(x_\sigma)+\nu (\sigma),\qquad \nu (\sigma)=\indic_{\{S_\sigma''(T_\sigma)<0\}},
\end{equation}
so local minimality can be destroyed through exactly two mechanisms: (1) the \emph{duration mechanism}: $\nu$ jumps from $0$ to $1$, i.e.\ $S_\sigma''(T_\sigma)$ changes sign from $+$ to $-$ and (2) the \emph{conjugate mechanism}: $\mm_{T_\sigma}$ jumps up, governed by Theorem~\ref{thm:D}.

\begin{theorem}[Free-time local stability]\label{thm:Dfree}
Let $(x_\sigma,T_\sigma)$ be a $C^1$ branch of free-time critical points such that $(x_{\sigma_0},T_{\sigma_0})$ is a local minimizer, hence $\mm_{T_{\sigma_0}}=0$, $\nu (\sigma_0)=0$.
\begin{enumerate}
\item[(1)] \emph{(Duration mechanism.)}  If for every $\delta>0$ there exists $\sigma\in(\sigma_0,\sigma_0+\delta)$ with $S_\sigma''(T_\sigma)<0$ (i.e.\ $\nu (\sigma)=1$)---in particular, if $S_\sigma''(T_\sigma)$ is $C^1$ and changes sign from positive to negative at $\sigma_0$---then $(x,T):=(x_{\sigma_0},T_{\sigma_0})$ is positively unstable: arbitrarily close parameters have $\mm(x_\sigma,T_\sigma)\geq1$.
\item[(2)] \emph{(Conjugate mechanism.)}  If $\dim\ker\Ac(\sigma_0,1)=1$ and the crossing curve satisfies $s'(\sigma_0)<0$, then for all sufficiently small $\sigma>\sigma_0$ one has $\mm_{T_\sigma}(x_\sigma)=1$, hence $\mm(x_\sigma,T_\sigma)\geq1$.
\item[(3)] \emph{(Robust case.)}  If $S_{\sigma_0}''(T_{\sigma_0})>0$ and either there is no crossing at $(\sigma_0,1)$ or the crossing curve satisfies $s'(\sigma_0)>0$, then there exists $\delta>0$ such that for all $\sigma\in(\sigma_0,\sigma_0+\delta)$: $\mm_{T_\sigma}(x_\sigma)=0$, $\nu(\sigma)=0$, and $(x_\sigma,T_\sigma)$ remains a local minimizer.
\end{enumerate}
\end{theorem}

\begin{proof}
The proof is immediate from the analysis in section \ref{subsec:thmD} and formula \eqref{eq:twomech}.
\end{proof}

\section{Example: quartic double well}\label{sec:examples}

\subsection{The quartic double well: noise lowers the conjugate threshold}\label{subsec:dw}

Consider the quartic double well
\begin{equation}\label{eq:dwV}
	V(x)=\frac{(x^2-1)^2}{4},\qquad x\in\R,
\end{equation}
with minima $x_\pm=\pm1$ (nondegenerate by $V''(\pm1)=2>0$) and barrier top at $x=0$.  The OM effective potential \eqref{eq:U} is
\begin{equation}\label{eq:dwU}
	U_\sigma(x)=\sigma(3x^2-1)-\tfrac12 x^2(x^2-1)^2.
\end{equation}
Now we summarize the properties of $U_\sigma$ as follows.
\begin{lemma}\label{lem:dwU}
	For $\sigma>0$:
	\begin{enumerate}[label=(\roman*)]
		\item\label{pp-i} $U_\sigma(\pm1)=2\sigma>0$; hence $k_0(L;x_\pm)\geq2\sigma>0$ and the free-time theory does not apply to the transition $-1\to+1$ (Lemma~\ref{lem:necessary}); one works with the fixed-time (truncated) functionals of Theorem~\ref{thm:A}\emph{(ii)}.
		\item\label{pp-iv} $U_\sigma'(x)=x\big[6\sigma-(x^2-1)(3x^2-1)\big]$; so critical points of $U_\sigma$ are $x=0$ and $x=\pm\sqrt{y_\pm}$, $y_\pm=\frac{2\pm\sqrt{1+18\sigma}}{3}$, where $y_+>0$ for all $\sigma>0$ while $y_-\geq0$ only for $\sigma\leq1/6$; at $\sigma=1/6$ the pair $\pm\sqrt{y_-}$ merges with the critical point $x=0$; the global maxima are at $\pm\sqrt{y_+}=\pm\big(1+\tfrac{3\sigma}{2}+O(\sigma^2)\big)$ with value $2\sigma+O(\sigma^2)$, so $c_u(L)=\sup U_\sigma=2\sigma+O(\sigma^2)$.
		\item\label{pp-ii} $U_\sigma''(x)=6\sigma-15x^4+12x^2-1$ and in particular $U_\sigma''(0)=6\sigma-1$, so the barrier top of $V$ is a local maximum of $U_\sigma$ for $\sigma<1/6$ and a local minimum for $\sigma>1/6$.
		\item\label{pp-iii} $\max_{x\in\R}U_\sigma''(x)=6\sigma+\tfrac75$, attained at $x=\pm\sqrt{2/5}$; $U_\sigma''$ is positive exactly on a $\sigma$-dependent neighborhood of $\pm\sqrt{2/5}$.	
	\end{enumerate}
\end{lemma}

%\begin{proof}
%	Direct computation: $\nabla V=x(x^2-1)$, $\Delta V=3x^2-1$, $U_\sigma=\sigma\Delta V-\frac12|\nabla V|^2$ gives \eqref{eq:dwU}; differentiating twice gives (ii)--(iii) (the maximum of the even quartic $U_\sigma''$ is at $x^2=2/5$).  For (iv), $U_\sigma'(x)=x\big[6\sigma-(x^2-1)(3x^2-1)\big]$, and solving $(y-1)(3y-1)=6\sigma$ for $y=x^2$ gives $y_\pm$; expansions at $\sigma=0$ give the asymptotics.  Part (i) is Lemma~\ref{lem:necessary} with $k=0<U(\pm1)\leq k_0(L;x_\pm)$.
%\end{proof}

By Theorem~\ref{thm:A}(ii), for every $\tau>0$ the truncated functional $\Lc^{\leq\tau}$ (with $k=0$) attains its minimum at duration $\tau$. Let $x_\sigma^\tau$ denote a fixed-time minimizer, which is a critical point of $\Lc^\tau(\sigma)$.  Its second variation is the Sturm--Liouville operator
\begin{equation}\label{eq:dwjacobi}
	\mathcal{J}_\sigma^\tau=-\frac{d^2}{dr^2}-U_\sigma''(x_\sigma^\tau(r)),\qquad r\in[0,\tau],
\end{equation}
with Dirichlet boundary conditions; $\mm_\tau(x_\sigma^\tau)$ is the number of negative eigenvalues of $\mathcal{J}_\sigma^\tau$, and the first conjugate time $\tau_c(\sigma)$ is the first $\tau$ for which $\ker\mathcal{J}_\sigma^\tau\neq\{0\}$.

In the Hamiltonian form of Section~\ref{subsec:hamiltonian}, this $C^1$ family (w.r.t $\sigma$) of paths $x_\sigma^\tau$ carries the two-parameter operator family
\begin{equation}\label{eq:dwA}
	\Ac(\sigma,s)=-J\frac{d}{dt}-B_{\sigma,s}(t),\qquad
	B_{\sigma,s}(t)=s\tau
	\begin{pmatrix}
		1 & 0\\
		0 & U_\sigma''(x_\sigma^\tau(st))
	\end{pmatrix},
\end{equation}
on $W^{1,2}_D$ (for the mechanical Lagrangian $P=I$, $Q=0$, $R=-U_\sigma''$ in \eqref{eq:B}).  By \eqref{eq:kernelcorresp} its crossing instants $(\sigma,s)$ are exactly the conjugate instants of the $s$-rescaled problem.

Note that the graph of $x_\sigma^\tau$ is the same of the straight path 
\begin{equation*}
	\bar x^\tau(r):=-1+2r/\tau,
\end{equation*} which is regarded as an approximation of $x_\sigma^\tau$ in the following crossing-instant analysis.
 Along the straight path $\bar x^\tau$---a $\sigma$-independent family---one has $\partial_\sigma U_\sigma''\equiv6$, hence for $u=(y,\xi)\in\ker\Ac(\sigma,s)$,
\begin{equation}\label{eq:dwcross}
	\begin{aligned}
		\Gamma(\Ac(\sigma,s),\sigma)[u]
	&=-\angles{\partial_\sigma B_{\sigma,s}\,u}{u}_{L^2}\\
	&=-s\tau\int_0^1\partial_\sigma U_\sigma''\big(\bar x^\tau(st)\big)\,\xi(t)^2\,dt
	=-6s\tau\int_0^1\xi(t)^2\,dt<0.	
	\end{aligned}
\end{equation}
The slope formula \eqref{eq:sprime} then gives
\begin{equation}\label{eq:dwdown}
	s'(\sigma)=-\frac{\Gamma(\Ac,\sigma)[u]}{\Gamma(\Ac,s)[u]}<0,
\end{equation}
so all its crossing curves are strictly decreasing and enter the interval $(0,1]$ from the right as $\sigma$ increases, which is the noise-sensitive alternative of Theorem~\ref{thm:D}(2). 

 Along a true extremal branch $x_\sigma^\tau$ the $\sigma$-variation of the orbit contributes the additional term $-s\tau\int_0^1U_\sigma'''(x_\sigma^\tau(st))\,\partial_\sigma x_\sigma^\tau(st)\,\xi(t)^2\,dt$ to \eqref{eq:dwcross}, whose sign is checked numerically in Figure~\ref{fig:crossings} and characterized at the first crossing in Proposition~\ref{prop:pitchfork}.

\begin{figure}[t]
	\centering
	\includegraphics[width=0.9\textwidth]{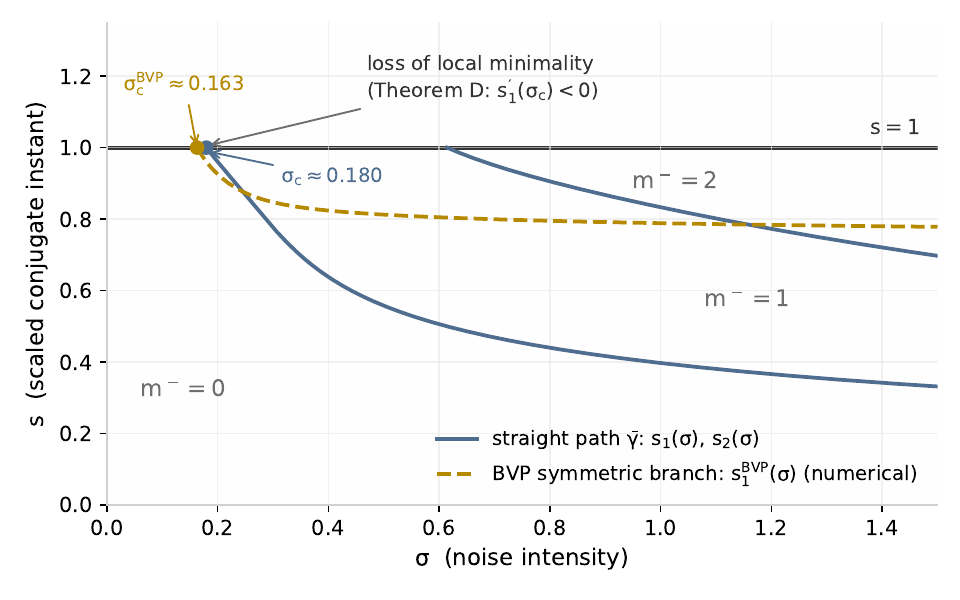}
	\caption{Crossing curves in the $(\sigma,s)$ plane for the quartic double well \eqref{eq:dwV}--\eqref{eq:dwU} at fixed horizon $\tau=3$.  Each curve $s_m(\sigma)=r_m(\sigma)/\tau$ is a conjugate instant of the $s$-rescaled problem, i.e.\ a positive zero of the Jacobi field $\ddot\xi=-U_\sigma''(\gamma(r))\xi$, $\xi(0)=0$, $\dot\xi(0)=1$; the index of the Dirichlet problem on $[0,s]$ equals the number of curves below the level $s$ (regions labelled $\mm=0,1,2$).  \emph{Solid blue:} the straight path $\bar\gamma(r)=-1+2r/\tau$ (an approximation to the direct transition path).  \emph{Dashed ochre:} the symmetric branch obtained by solving the Euler--Lagrange boundary value problem $\ddot\gamma=-U_\sigma'(\gamma)$, $\gamma(0)=-1$, $\gamma(\tau)=1$ numerically (\texttt{solve\_bvp} with continuation in $\sigma$).  When a curve crosses the emphasized line $s=1$ \emph{downward}, the fixed-time path acquires exactly one negative direction: this is the loss of local minimality of Theorem~\ref{thm:D}(2), and the sign of the slope at the crossing, $s_1'(\sigma_c)<0$, is precisely the stability criterion of Theorem~\ref{thm:D}.  The crossing occurs at $\sigma_c\approx0.180$ for the straight path and at $\sigma_c^{\mathrm{BVP}}\approx0.163$ for the symmetric BVP branch.}
	\label{fig:crossings}
\end{figure}

Figure~\ref{fig:crossings} displays the crossing curves of Section~\ref{subsec:crossingcurve} for the horizon $\tau=3$.  For the straight path $\bar x^\tau$, the Jacobi equation $\ddot\xi=-U_\sigma''(\bar x^\tau(r))\xi$ with $\xi(0)=0$, $\dot\xi(0)=1$ was integrated on a grid of $241$ values of $\sigma\in[0,1.5]$.
The same computation was then repeated along the true extremal: for each $\sigma$ the boundary value problem $\ddot\gamma=-U_\sigma'(\gamma)$, $\gamma(0)=-1$, $\gamma(\tau)=1$ was solved by collocation with continuation in $\sigma$, yielding the symmetric branch whose first conjugate instant $s_1^{\mathrm{BVP}}(\sigma)$ is shown dashed. The proximity of the two crossing curves quantifies the sense in which the straight path is a good approximation to the direct transition path at moderate noise.

\subsection{The fixed-time bifurcation diagram: pitchfork and global switching}\label{subsec:dwshoot}

Figure~\ref{fig:crossings} in fact tracks the crossing (or conjugate) instants of the \emph{symmetric} branch only.  We now compute more solutions of the fixed-time Euler--Lagrange boundary value problem
\begin{equation}\label{eq:dwbvp}
	\ddot\gamma=-U_\sigma'(\gamma),\qquad \gamma(0)=-1,\quad \gamma(\tau)=1,\qquad \tau=3,
\end{equation}
by shooting on the initial velocity $v=\dot\gamma(0)$ (roots of $\gamma(\tau;v)=1$), recording for each
branch the action $\Lc^\tau(\gamma)$ and the Morse index $\mm_\tau$; see Figure~\ref{fig:bifurcation}.  The problem is $\mathbb{Z}_2$-symmetric under $\gamma(t)\mapsto-\gamma(\tau-t)$, so
non-symmetric solutions come in mirror pairs.

\begin{figure}[t]
	\centering
	\includegraphics[width=\textwidth]{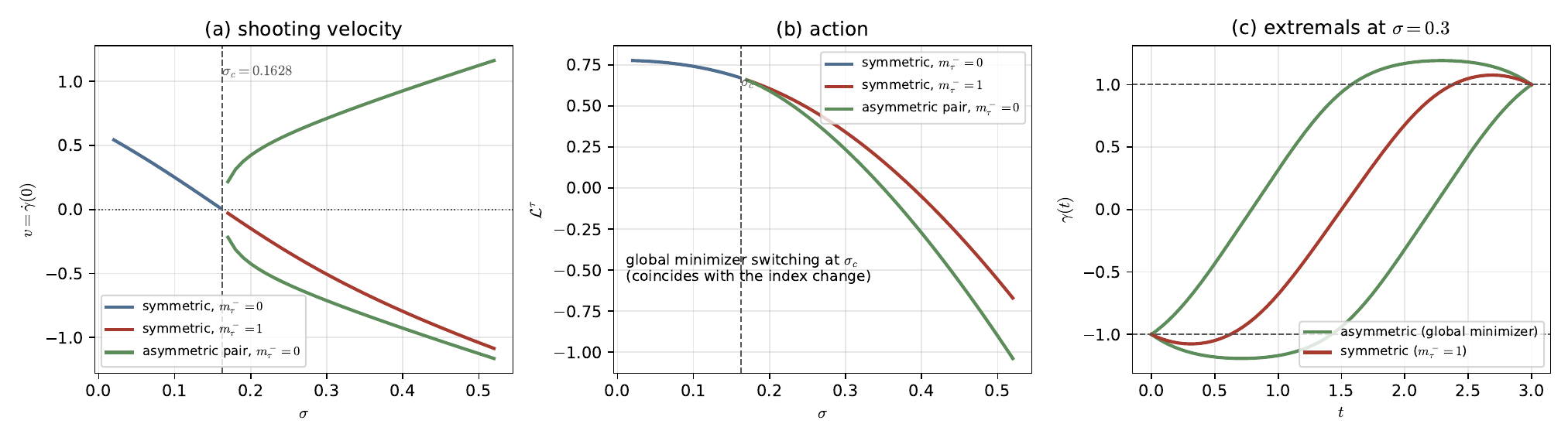}
	\caption{The fixed-time bifurcation diagram of the double-well problem \eqref{eq:dwbvp} at $\tau=3$
		(via shooting method).  (a) Shooting velocity $v=\dot\gamma(0)$ versus $\sigma$;
		(b) action $\Lc^\tau$ versus $\sigma$; (c) the three coexisting extremals at $\sigma=0.3$.  Below
		$\sigma_c=0.1628$ the symmetric branch is the unique critical point ($\mm_\tau=0$) among numerical observation; at $\sigma_c$ it undergoes
		a supercritical pitchfork, and a mirror pair of asymmetric minimizers is born
		with strictly smaller action.  The minimizer switching and the index change coincide.}
	\label{fig:bifurcation}
\end{figure}
\begin{proposition}[numerical result]\label{prop:pitchfork}
	For \eqref{eq:dwbvp} with $\tau=3$, on the computed range $\sigma\in[0.02,0.52]$:
	\begin{enumerate}
		\item[(i)] for $\sigma<\sigma_c=0.162826$, \eqref{eq:dwbvp} has a \emph{unique} solution in the searched range, which is the
		symmetric branch with $\mm_\tau=0$;
		\item[(ii)] at $\sigma=\sigma_c$ the smallest Dirichlet eigenvalue of \eqref{eq:dwjacobi} along the symmetric
		branch crosses zero---consistently with $\sigma_c^{\mathrm{BVP}}\approx0.163$ of
		Figure~\ref{fig:crossings}---and a \emph{supercritical pitchfork} occurs: for $\sigma>\sigma_c$ there are
		exactly three solutions, the symmetric branch with $\mm_\tau=1$ and a mirror pair of asymmetric branches with
		$\mm_\tau=0$;
		\item[(iii)] For $\sigma>\sigma_c$, furthermore, the asymmetric pair has \emph{strictly smaller action} than the symmetric branch for every
		computed $\sigma>\sigma_c$: e.g.\ $0.6588$ vs.\ $0.6596$ at $\sigma=0.17$, $0.4356$ vs.\ $0.4905$ at
		$\sigma=0.25$, $-0.9004$ vs.\ $-0.5569$ at $\sigma=0.5$.
	\end{enumerate}

\end{proposition}
	Consequently, the \emph{numerically global} minimizer of the fixed-time OM action switches from the symmetric to the
asymmetric branches exactly at $\sigma=\sigma_c$, at which index bifurcation and minimizer
switching coincide; see
Figure~\ref{fig:bifurcation}.
Proposition \ref{prop:pitchfork}(iii) shows that
for $\sigma>\sigma_c$ the most probable fixed-time transition path is one of the two
	asymmetric \emph{detour} paths, which cross the barrier region early (respectively late) and overshoot
	towards the maxima of $U_\sigma$ beyond $\pm1$ (Figure~\ref{fig:bifurcation}(c)), while the symmetric
	barrier-crossing path (with $m_\tau^-=1$) survives as a saddle of the action.

\subsection{The \texorpdfstring{$\{0,1\}$}{(0,1)}-index in an OM potential: a numerical example}\label{subsec:dwfree}

 We now exhibit the case
\[
\mm_T=0,\qquad \nu=1,\qquad \mm=1
\]
of Theorem~\ref{thm:B} in the quartic double well \eqref{eq:dwV} with effective potential $U_\sigma$
\eqref{eq:dwU}. 

The example with above property is a saddle point (never a minimizer) of the

Fix $\sigma=0.3$.  Then $U_\sigma(0)=-0.3$ and $U_\sigma''(0)=0.8>0$, and the connected component of
$\{U_\sigma<0\}$ containing the origin is $(-a,a)$, $a=0.642212$.  Choose
\[
x_-=0,\qquad x_+=0.8\,a=0.513770,
\]
both in the interior of $\{U_\sigma\leq0\}$. Hence $k_0(L;x_\pm)<0$ and the free-time theory of
Sections~\ref{sec:01}--\ref{sec:stab} applies; see Lemma \ref{lem:necessary}.  

Zero-energy extremals satisfy
$\tfrac12\dot\gamma^2=-U_\sigma(\gamma)$ and are therefore segments of the unique periodic orbit of energy
$0$, which oscillates between the turning points $\pm a$ with quarter period
\begin{align}\label{quarter-t}
	t_0:=\int_0^a\frac{dx}{\sqrt{-2U_\sigma(x)}}=1.202915.
\end{align}
This extremal is exactly the brake orbit with zero-velocity at points $\pm a$; see e.g. \cite{liu_seifert_2014,zhang_maslov-type_2025}. 

Between $x_-$ and $x_+$ there are two distinguished zero-energy connections (Figure~\ref{fig:dw01}(a)--(b)):
the \emph{short arc} $0\to x_+$ of duration $T=0.735336$, and the \emph{long arc} $0\to a\to x_+$, which first
climbs to the turning point and falls back, of duration $T=1.670495$.

\begin{figure}[t]
	\centering
	\includegraphics[width=\textwidth]{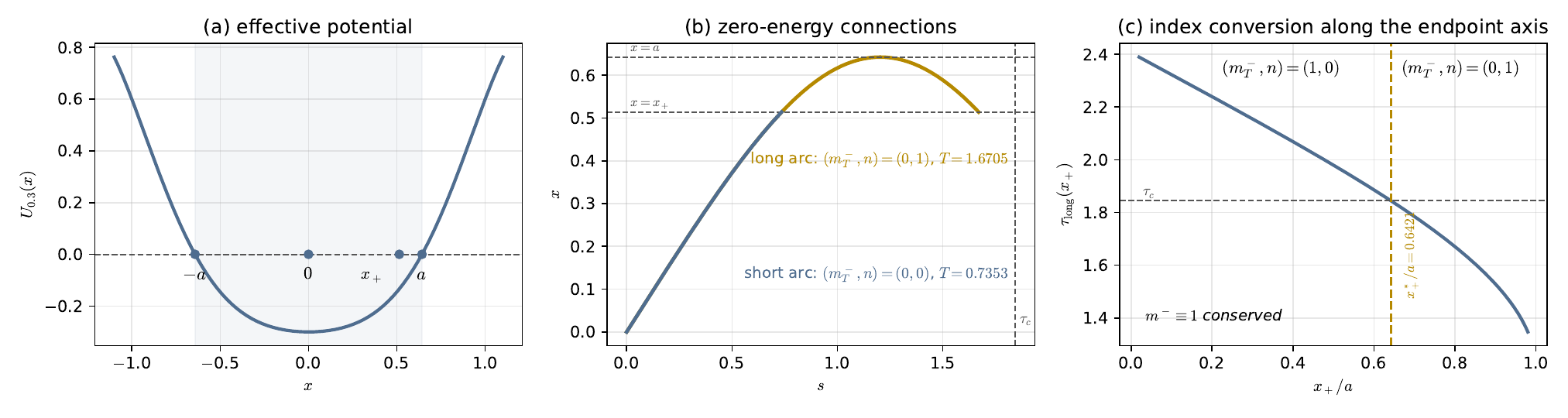}
	\caption{The $\{0,1\}$-index in the double-well OM model at $\sigma=0.3$ (Section~\ref{subsec:dwfree}).
		(a) The effective potential $U_{0.3}$ and its sublevel well $(-a,a)$, $a=0.6422$, containing the endpoints
		$x_-=0$, $x_+=0.8a$.  (b) The two zero-energy connections in the $(s,x)$ plane: the short arc
		($T=0.7353$, $(\mm_T,n)=(0,0)$) and the long arc ($T=1.6705$, $(\mm_T,n)=(0,1)$); the dashed vertical line is
		the first conjugate time $\tau_c=1.8443$ of the orbit.  (c) Duration $\tau_{\mathrm{long}}(x_+)$ of the long
		arc versus $x_+/a$; it crosses $\tau_c$ at $x_+^*=0.6421\,a$, where the two terms of $\mm=\mm_T+\nu $ exchange
		($\Delta\mm_T=-1$, $\Delta \nu=+1$) with the sum $\mm\equiv1$ conserved (Proposition~\ref{prop:dwconv}).}
	\label{fig:dw01}
\end{figure}

\begin{proposition}[numerical result]\label{prop:dw01}
	For the double-well OM model \eqref{eq:dwV}--\eqref{eq:dwU} at $\sigma=0.3$ with $x_-=0$, $x_+=0.8a$ as above:
	\begin{enumerate}
		\item[(i)] the long arc is a \emph{nondegenerate fixed-time local minimizer}: the Dirichlet spectrum of the
		second-variation operator \eqref{eq:dwjacobi} along it starts at $\lambda_1=0.7109>0$; in particular
		$\mm_T=0$;
		\item[(ii)] the reduced action \eqref{eq:reduced} along the branch of fixed-time critical points through the long arc satisfies
		\[
		S''(T)=-2.322<0,\quad\text{equivalently}\quad C_0=0.1631<3.5050=\angles{\beta}{A^{-1}\beta},
		\]
		so $\nu=1$ and the free-time Morse index is $\mm=1$;
		\item[(iii)] the short arc has $(\mm_T,\nu)=(0,0)$, hence is a free-time local minimizer, with action
		$0.363058$ against $0.455019$ for the long arc.
	\end{enumerate}
Therefore, a path that is locally optimal at fixed duration can be a free-time saddle of index $1$ if
	prolonging it lowers the action. This suggests that the  mechanism of Theorem~\ref{thm:B} is realized with $\nu=1$ in this OM gradient model.
\end{proposition}

\begin{proof}[Numerical method]
	Orbits were integrated as $\ddot\gamma=-U_\sigma'(\gamma)$ with relative tolerance $10^{-11}$; durations from
	the quadrature $dx/\sqrt{-2U_\sigma(x)}$; the Dirichlet spectrum from a $4000$-point finite-difference
	discretization of \eqref{eq:dwjacobi}; the Schur complement from one solve of $Au=\beta$ with the same
	discretization; and $S(T)$ by collocation (\texttt{scipy.solve\_bvp}) with warm starts on
	$[T-0.15,T+0.15]$, where the identity $S'(T)=-E$ was verified pointwise and $S''$ read off from a linear fit
	of $E(T)$ at $E=0$.
\end{proof}

We also find that moving the endpoint $x_+$ reveals the two terms of the index identity $\mm=\mm_T+\nu $ exchanging roles.  All
long arcs $0\to a\to x_+$ are segments of one and the same zero-energy periodic orbit, so their first
	conjugate time measured from $x_-$ does not depend on $x_+$, and this conjugate time is $\tau_c=1.844315$. Note that $t_0$ \eqref{quarter-t} is the duration from $x_-=0$ to $a$. The duration
$\tau_{\mathrm{long}}(x_+)=t_0+\int_{x_+}^a(-2U_\sigma(x))^{-1/2}dx$ decreases strictly from $2t_0=2.405831$
to $t_0$ as $x_+$ moves from $0$ to $a$, and crosses $\tau_c$ at the unique point $x_+^*=0.412382$
($=0.6421\,a$); see Figure~\ref{fig:dw01}(c).

\begin{proposition}[index conversion]\label{prop:dwconv}
	For $\sigma=0.3$, $x_-=0$ and $x_+\in(0,a)$, the long arc satisfies
	\begin{align*}
		&	x_+<x_+^*\ (\text{ equivalently, }\tau_{\mathrm{long}}>\tau_c):\quad(\mm_T,\nu)=(1,0),\\
	&	x_+>x_+^*\ (\text{ equivalently, }\tau_{\mathrm{long}}<\tau_c):\quad(\mm_T,\nu)=(0,1),
	\end{align*}
	with $\mm=\mm_T+\nu \equiv1$ throughout. Namely, at $x_+^*$ the first conjugate point leaves the interval
	$[0,\tau_{\mathrm{long}}]$ while the duration direction switches on ($\Delta\mm_T=-1$, $\Delta \nu=+1$,
	$\Delta\mm=0$ as $x_+$ increases through $x_+^*$).  This is the conversion mechanism underlying
	Theorem~\ref{thm:C}, observed here along the endpoint axis.
\end{proposition}

\begin{remark}\label{rmk:realans}
	Proposition~\ref{prop:dw01} shows that the
	$\{0,1\}$-alternative with $\nu=1$ occurs in a genuine OM
	potential. 
We remark two features of the endpoints used here.
First, they are \emph{not} the wells of $V$, where
$U_\sigma(\pm1)=2\sigma>0$ for every $\sigma>0$ and hence $k_0(L;\pm1)>0$; only the
fixed-time theory applies (Lemma~\ref{lem:dwU}(i)). In the present example $x_-=0$ is the barrier top of $V$ (a local minimum of
$U_{0.3}$) and $x_+=0.8a$ lies on the inner shoulder of the right well. 
 Second, well-adjacent endpoints
are nevertheless admissible at small noise.  Indeed, choose $x=1+u$ and
$U_\sigma(1+u)=2\sigma-2u^2+O(u^3+\sigma|u|)$ as $u,\sigma\to0$, so the inner zeros of $U_\sigma$ sit at
$\pm a(\sigma)$ with $a(\sigma)=1-\sqrt{\sigma}+o(\sqrt{\sigma})$, and the middle component of
$\{U_\sigma<0\}$ is $(-a(\sigma),a(\sigma))$; hence for every fixed $\delta>0$ the symmetric endpoints
$x_\pm=\pm(1-\delta)$ satisfy $k_0(L;x_\pm)<0$ whenever $\sigma$ small enough (so that the whole segment
$[x_-,x_+]$ lies in $\{U_\sigma<0\}$).  In other words, each well ($\pm1$) carries an exclusion zone of radius
$O(\sqrt{\sigma})$, shrinking with the noise, outside of which endpoints belong to the free-time theory.
Physically, therefore, free-time extremals do describe transition paths between neighborhoods of the two
wells---as critical points of $\Lc_0$, whose local minimality is governed by the index theory of
Sections~\ref{sec:second}--\ref{sec:stab}.
\end{remark}

\appendix

\renewcommand\thesection{\Alph{section}}
\section{Completeness and the Palais--Smale condition}\label{app:ps}

In this appendix we prove Lemma~\ref{lem:ps}; the completeness statements of Lemma~\ref{lem:complete} are proved in the main text.  The argument is the classical one (cf.\ \cite[Lemma~5.3]{abbondandolo_lectures_2013-1}, \cite[Lemma~2.5]{asselle_existence_2016}), simplified by the mechanical form \eqref{eq:L} of our Lagrangian, and adapted to the noncompact configuration space $\R^n$ via the two-sided estimate \eqref{eq:twosided}.

\begin{proof}[Proof of Lemma~\ref{lem:ps}]
Let $\{(x_n,T_n)\}\subset\Wc$ be a (PS) sequence for $\Lc_k$ at level $c$, i.e.\ $\Lc_k(x_n,T_n)\to c$ and $\|d\Lc_k(x_n,T_n)\|\to0$.

\emph{Step 1: bounds on $T_n$ and $\dot x_n$.}  From \eqref{eq:twosided},
\[
\frac{l^2}{2T_n}+T_n\big(k-c_u(L)\big)\leq c+1,\qquad l:=l(x_\pm)=|x_+-x_-|>0.
\]
If $k>c_u(L)$, this gives $T_n\in[T_{\min},T_{\max}]$ with $T_{\min}:=l^2/(2(c+1))$ and $T_{\max}:=(c+1)/(k-c_u(L))$.  In the truncated case (ii), $T_n\leq\tau$ by definition and
\[
\frac{l^2}{2T_n}\leq c+1+T_n\big(c_u(L)-k\big)\leq c+1+\tau\big|c_u(L)-k\big|=:C',
\]
so $T_n\geq l^2/(2C')>0$ .  Up to a subsequence, $T_n\to T_*>0$ (with $T_*\leq\tau$ in case (ii)).  Moreover,
\[
\frac{1}{2T_n}\int_0^1|\dot x_n|^2dt\leq c+1+T_n\big(c_u(L)-k\big),
\]
so $\|\dot x_n\|_{L^2}\leq C$ uniformly.

\emph{Step 2: confinement to a bounded region.}  Since $x_n(0)=x_-$ and $\|\dot x_n\|_{L^2}\leq C$,
\[
|x_n(t)-x_n(t')|\leq C|t-t'|^{1/2},\qquad |x_n(t)|\leq|x_-|+C=:r,
\]
so $\{x_n\}$ is equicontinuous and uniformly bounded in the closed ball $\bar B_r(0)\subset\R^n$.  By Arzel\`a--Ascoli, a subsequence converges uniformly to a continuous $x$, and $x\in\Wpm$ since a further subsequence converges weakly in $W^{1,2}$.  All data below ($U$, $\nabla U$) are now only needed on $\bar B_r(0)$---this is where the noncompactness of $\R^n$ is compensated; note that \ref{ass-h2} is used precisely to obtain the two-sided bound of Step~1.

\emph{Step 3: strong convergence.}  Let $g_n:=d_x\Lc^{T_n}(x_n)\in(\Wz)^*$; then $\|g_n\|\to0$.  For the mechanical Lagrangian,
\[
g_n[\xi]=\int_0^1\Big[\frac{1}{T_n}\angles{\dot x_n}{\dot\xi}-T_n\angles{\nabla U(x_n)}{\xi}\Big]dt.
\]
Testing the difference with $\xi=x_n-x$:
\[
\begin{aligned}
&g_n[x_n-x]-d_x\Lc^{T_n}(x)[x_n-x]\\
&\qquad=\frac{1}{T_n}\int_0^1|\dot x_n-\dot x|^2dt-T_n\int_0^1\angles{\nabla U(x_n)-\nabla U(x)}{x_n-x}dt.
\end{aligned}
\]
The left-hand side tends to $0$: $g_n[x_n-x]\to0$ since $\|g_n\|\to0$ and $\|x_n-x\|_{W^{1,2}}$ is bounded, and $d_x\Lc^{T_n}(x)[x_n-x]\to0$ since $x_n\rightharpoonup x$ weakly in $W^{1,2}$ and $T_n\to T_*$.  The second term on the right tends to $0$ since $x_n\to x$ uniformly and $\nabla U$ is bounded on $\bar B_r(0)$.  Hence $\|\dot x_n-\dot x\|_{L^2}\to0$, i.e.\ $x_n\to x$ strongly in $\Wpm$, and $(x_n,T_n)\to(x,T_*)$ in $\Wc$.

\emph{Step 4: conclusion.}  By continuity, $\Lc_k(x,T_*)=c$ and $d\Lc_k(x,T_*)=0$ (in case (ii), $d_x\Lc(x,T_*)=0$ with $T_*\leq\tau$).  This proves (i) and (ii).
\end{proof}

\begin{remark}\label{rmk:psvariants}
The same argument shows that $\Lc_k$ (resp.\ $\Lc_k^{\leq\tau}$) satisfies the Palais--Smale condition \emph{at every level $c$} under (H1)--(H2) as soon as the a priori bound on $\{T_n\}$ from Step~1 is available. 
% At the critical level $k=c_u(L)$ the upper bound on $T_n$ is lost (the term $T(k-c_u)$ vanishes), which is exactly the obstruction behind Theorem~\ref{thm:A}(iii).
\end{remark}

\section{Boundary representation of crossing forms}\label{app:crossing}

We give the crossing-form computation used in Lemma~\ref{lem:scrossing} and Proposition~\ref{prop:crossingcurve}, following the standard argument (cf.\ \cite{cox_hamiltonian_2023,portaluri_linear_2021}) and harmonizing the signs with our convention $\Gamma(\Ac,\lambda)=\angles{\partial_\lambda\Ac\,\cdot}{\cdot}|_{\ker}$.

\begin{theorem}\label{thm:boundary}
Let $[0,1]\ni\lambda\mapsto B_\lambda(\cdot)$ be a $C^1$ family of continuous symmetric $2n\times2n$ matrix functions, let
\[
\Ac_\lambda=-J\frac{d}{dt}-B_\lambda(t):\ W^{1,2}_D\to L^2(\II,\R^{2n}),
\]
and let $\phi_\lambda:\II\to\Spm(2n)$ be the fundamental solution of $\phi_\lambda'=JB_\lambda\phi_\lambda$, $\phi_\lambda(0)=I_{2n}$.  Then for every crossing $\lambda_*$ and all $u,v\in\ker\Ac_{\lambda_*}$,
\begin{equation}\label{eq:boundary}
\begin{aligned}
\Gamma(\Ac,\lambda_*)[u,v]
&=-\int_0^1\angles{\partial_\lambda B_{\lambda_*}(t)u(t)}{v(t)}dt\\
&=\big\langle J\phi_{\lambda_*}^{-1}(1)\,\partial_\lambda\phi_{\lambda_*}(1)\,u(0),\ v(0)\big\rangle.
\end{aligned}
\end{equation}
\end{theorem}

\begin{proof}
The first identity is the definition \eqref{eq:crossingform} with $\partial_\lambda\Ac_\lambda=-\partial_\lambda B_\lambda$.  For \eqref{eq:boundary}, abbreviate $\phi=\phi_{\lambda_*}$, $B=B_{\lambda_*}$ and write $u(t)=\phi(t)u_0$, $v(t)=\phi(t)v_0$ with $u_0=u(0)$, $v_0=v(0)$ (elements of the kernel are determined by their initial data since $\Ac_{\lambda_*}u=0$ is the equation $\dot u=JBu$).  Define
\[
F(t):=\big\langle \partial_\lambda\phi(t)\,u_0,\ J\phi(t)\,v_0\big\rangle,
\]
where $\partial_\lambda\phi$ is the $\lambda$-derivative at $\lambda_*$.  Differentiating $\phi_\lambda'=JB_\lambda\phi_\lambda$ in $\lambda$ gives $(\partial_\lambda\phi)'=J(\partial_\lambda B)\phi+JB\,\partial_\lambda\phi$.  Hence, using $J^\top=-J$, $J^2=-I$, the symmetry of $B$, and $\phi'=JB\phi$:
\[
\begin{aligned}
F'(t)&=\big\langle J(\partial_\lambda B)\phi u_0+JB\,\partial_\lambda\phi u_0,\ J\phi v_0\big\rangle
+\big\langle \partial_\lambda\phi u_0,\ J(JB\phi)v_0\big\rangle\\
&=\angles{\partial_\lambda B\,u}{v}
+\big\langle JB\,\partial_\lambda\phi u_0,\ J\phi v_0\big\rangle
-\big\langle \partial_\lambda\phi u_0,\ Bv\big\rangle\\
&=\angles{\partial_\lambda B\,u}{v}
+\big\langle B\,\partial_\lambda\phi u_0,\ \phi v_0\big\rangle
-\big\langle \partial_\lambda\phi u_0,\ Bv\big\rangle
=\angles{\partial_\lambda B\,u}{v}.
\end{aligned}
\]
Since $\phi_\lambda(0)=I_{2n}$ for all $\lambda$, $\partial_\lambda\phi(0)=0$, so $F(0)=0$ and
\[
\int_0^1\angles{\partial_\lambda B\,u}{v}dt=F(1)-F(0)
=\big\langle \partial_\lambda\phi(1)u_0,\ J\phi(1)v_0\big\rangle.
\]
Finally, $\phi^\top J\phi=J$ (symplecticity) gives $\phi^\top J=J\phi^{-1}$, hence
\[
\begin{aligned}
\big\langle \partial_\lambda\phi(1)u_0,\ J\phi(1)v_0\big\rangle
&=-\big\langle \phi^\top(1)J\,\partial_\lambda\phi(1)u_0,\ v_0\big\rangle\\
&=\big\langle-J\phi^{-1}(1)\partial_\lambda\phi(1)u_0,\ v_0\big\rangle,
\end{aligned}
\]
which is \eqref{eq:boundary}.
\end{proof}

\begin{remark}\label{rmk:boundarysigns}
For the $s$-direction of Section~\ref{sec:sf}, $\partial_s\phi_{\sigma,s}(1)=JB_\sigma(s)\phi_{\sigma,s}(1)$ (since $\phi_{\sigma,s}(t)=\phi_\sigma(st)$), so \eqref{eq:boundary} gives
\[
\begin{aligned}
\angles{\partial_sB_{\sigma,s}u}{v}_{L^2}
&=\angles{-J\phi_{\sigma,s}^{-1}(1)\,JB_\sigma(s)\phi_{\sigma,s}(1)u(0)}{v(0)}\\
&=\angles{\phi_{\sigma,s}^\top(1)B_\sigma(s)\phi_{\sigma,s}(1)u(0)}{v(0)},
\end{aligned}
\]
because $-J\phi^{-1}J=\phi^\top$.  The crossing form \eqref{eq:crosss} carries the additional overall minus sign from $\partial_s\Ac=-\partial_sB_{\sigma,s}$; this is the source of the sign in Lemma~\ref{lem:scrossing}.
\end{remark}

%\section*{Acknowledgments}
%The first author was partly supported by the NSFC grant 12401129.
%The second author was partly supported by the NSFC grant 12141107,
%the Natural Science Foundation of Guangdong (Grant No.\
%2025A1515011188), the Cross Disciplinary Research Team on Data
%Science and Intelligent Medicine (2023KCXTD054), and the
%Guangdong--Dongguan Joint Research Fund (Grant 2023A1515140016).

\end{document}